\documentclass[hidelinks, 11pt, oneside, reqno]{amsart}
\usepackage{dsliheader}
\usepackage{cite}
\usepackage{caption}
\newtheoremstyle{named}{}{}{\itshape}{}{\bfseries}{.}{.5em}{\thmnote{#3}}
\theoremstyle{named}

\newcommand{\td}[1]{\widetilde{#1}}

\title[An Intersection Principle for Mean Curvature Flow]{An Intersection Principle for Mean Curvature Flow}
\author[Lee]{Tang-Kai Lee}
\author[Payne]{Alec Payne}

\address{Massachusetts Institute of Technology, 77 Massachusetts Avenue, Cambridge, MA 02139, USA}
\email{tangkai@mit.edu}

\address{North Carolina State University, 2311 Stinson Drive, Raleigh, NC 27607, USA}
\email{ajpayne4@ncsu.edu}

\begin{document}

\begin{abstract}
The avoidance principle says that mean curvature flows of hypersurfaces remain disjoint if they are disjoint at the initial time. We prove several generalizations of the avoidance principle that allow for intersections of hypersurfaces. First, we prove that the Hausdorff dimension of the intersection of two mean curvature flows is non-increasing over time, and we find precise information on how the dimension changes. We then show that the self-intersection of an immersed mean curvature flow has non-increasing dimension over time. Next, we extend the intersection dimension monotonicity to Brakke flows and level set flows which satisfy a localizability condition, and we provide examples showing that the monotonicity fails for general weak solutions. We find a localization result for level set flows with finitely many singularities, and as a consequence, we obtain a fattening criterion for these flows which depends on the behavior of intersections with smooth flows.
\end{abstract}
\maketitle
\vspace{-.25in}

\section{Introduction}
Mean curvature flow (MCF) provides a way to deform submanifolds in a canonical way, and as a result, it has many applications in geometry and topology.
Perhaps the most important property of codimension one MCF is that it satisfies the avoidance principle, which says that two smooth mean curvature flows of hypersurfaces remain disjoint if they are disjoint initially. The avoidance principle is ubiquitously used throughout the MCF literature, especially for controlling the location of an MCF by comparison with a well-chosen disjoint flow. However, very little is known about how to compare two MCFs if they are not disjoint.

In this paper, we describe the general behavior of the Hausdorff dimension and measure of the intersection of MCFs, both before and after the first singular time.
Our first result is that the dimension of the intersection of two properly embedded smooth MCFs is non-increasing in time if one of the flows is compact. 
Moreover, we find that the $(n-1)$-dimensional Hausdorff measure, $\cH^{n-1}$, of the intersection is finite after the initial time, and if the flows do not become instantaneously disjoint, then the $\cH^{n-1}$-measure remains positive strictly before the flows do become disjoint. We refer to this behavior, encapsulated in the conclusion of Theorem \ref{theorem main}, as the ``intersection principle." 

\begin{thm}\label{theorem main}
    Let $M$ and $N$ be complete, connected, smooth, and properly embedded hypersurfaces in $\mathbb{R}^{n+1}$ such that $M \neq N$ and at least one of these hypersurfaces is closed. Let $M_t$ and $N_t$ be smooth, proper\footnote{By this, we mean that the spacetime map defining the flow $F: M \times [0, T) \to \bR^{n+1}$ is continuous and proper, and each $F(\cdot, t)$ is a proper embedding (see~\cite[Remark 1.2, Lemma C.1]{PeacheyNonuniqueness}).} mean curvature flows starting from $M$ and $N$ which exist for $t\in [0, T)$.     
    
    Then, $\cH^{n-1}(M_t \cap N_t) < \infty$ for all $t \in (0,T)$, and the Hausdorff dimension of $M_t \cap N_t$ is non-increasing for $t\in [0, T)$. Moreover, if $\cH^{n-1}(M_s \cap N_s) =0$ for some $s \in [0, T)$, then $M_t \cap N_t = \emptyset$ for all $t \in (s,T)$.
    Specifically, there exists $t_0 \in [0,T]$ such that $0<\cH^{n-1}(M_t\cap N_t) < \infty$ for all $t \in (0,t_0)$ and $M_t \cap N_t = \emptyset$ for all $t \in (t_0, T)$. 
    If $t_0$ equals $0$ or $T$, the intervals $(0,0)$ or $(T,T)$ are interpreted as empty. 
\end{thm}

We note that the classical avoidance principle is a particular case of Theorem \ref{theorem main}. 
We assume that each of the flows exists for some time $T$, since we make no assumptions on the boundedness of the curvature of $M$ or $N$. This assumption is automatic when both $M$ and $N$ are compact.

Theorem~\ref{theorem main} is known in the one-dimensional case by Angenent~\cite{Angenent91} and in the rotationally symmetric case by Altschuler--Angenent--Giga~\cite{AAG95}.
Both of these results use Sturmian theory for solutions to one-dimensional linear parabolic differential equations~\cite{Ang88}. 
Sturmian theory does not directly generalize to higher dimensions, but a recent beautiful result of Huang--Jiang gives a partial generalization~\cite{HJ} (see~\cite{HanLin} and references in \cite{HanLin, HJ} for earlier work). 
Huang--Jiang's result provides estimates on the measure and dimension of zero sets of scalar-valued linear parabolic PDE with time-dependent coefficients, and we apply this result in our proof of Theorem \ref{theorem main} to find the finiteness of the $\cH^{n-1}$-measure of the intersection of MCFs. 
We next show that if two smooth $n$-dimensional MCFs in $\bR^{n+1}$ intersect on a set of $\cH^{n-1}$-measure zero, then they must lie on one side of each other (see Section~\ref{sub:oneside}). 
We then use arguments based on the strong maximum principle to find the precise behavior described in Theorem \ref{theorem main}.

Our result in Theorem~\ref{theorem main} is the sharpest way to generalize the one-dimensional Sturmian theory to the MCF of hypersurfaces. In~\cite{Angenent91}, Angenent showed that if there are two smooth curve shortening flows (CSFs) starting from simple closed curves, then for positive time, the intersection of the two curves is a finite set and the number of intersection points is non-increasing.
In Theorem \ref{theorem main}, Angenent's finiteness result of the intersection of CSFs is generalized to a finiteness of the codimension two Hausdorff measure of the intersection. 
Regarding the monotonicity of the number of intersections of CSFs, there are several possible generalizations one could consider for $n$-dimensional MCF in $\bb R^{n+1}.$ 
One may consider a potential monotonicity of either the $\cH^{n-1}$-measure of the intersection of the flows or a potential monotonicity of the number of connected components of the intersection of the flows. 
In Section~\ref{sec:ApplicationsandExamples}, we show that both of these fail for MCF in $n>1$ (see Propositions~\ref{proposition increasing Hn-1 all time}, ~\ref{proposition increasing Hn-1 short time}, and~\ref{proposition nodal domain monotonicity failure}), even for compact mean convex MCF. 
Our examples are analogous to those of Huang--Jiang, who showed that the $\cH^{n-1}$-measure of the zero set of a linear parabolic PDE on $\bR^n$ need not be monotone in time~\cite[Example 1.17]{HJ}.

We next find that the intersection principle can be extended to the self-intersection set of an immersed mean curvature flow. 
The immersed case has unique challenges beyond the embedded case, which include proving that an immersion with an $\cH^{n-1}$-measure zero self-intersection set can be perturbed to an embedding (see Proposition \ref{perturb-immersed-to-embedded}).

\begin{thm}\label{thm:self-int-main}
    Let $M$ be a smooth, closed, connected $n$-manifold, and let $F_t: M \to \bR^{n+1}$ be an immersed mean curvature flow which exists for $t \in [0, T)$.
    Then, the self-intersection set $S_t$ of $M_t$ satisfies $\cH^{n-1}(S_t) < \infty$ for all $t \in (0,T)$, and the Hausdorff dimension of $S_t$ is non-increasing for $t\in [0, T)$. Moreover, if $\cH^{n-1}(S_s)=0$, then $F_t$ is an embedding for all $t \in (s, T)$.
    Specifically, there exists $t_0 \in [0,T]$ such that $0<\cH^{n-1}(S_t)< \infty$ for all $t \in (0,t_0)$ and $S_t = \emptyset$ for all $t\in (t_0, T)$. If $t_0$ equals $0$ or $T$, we interpret the intervals $(0,0)$ and $(T,T)$ as empty.
\end{thm}

Theorem \ref{thm:self-int-main} implies that an $n$-dimensional immersed MCF in $\bR^{n+1}$ becomes instantaneously embedded if the self-intersection set has $\cH^{n-1}$-measure zero. The analogous statement for higher codimension MCF does not hold. For example, one could consider the MCF of an immersed $2$-submanifold in $\bR^{n+1}$, $n\geq 3$, with a single self-intersection point locally modeled by a pair of $2$-planes intersecting at one point. 

As an application of Theorem \ref{thm:self-int-main}, if $M$ is a closed $n$-manifold which may not be embedded in $\bR^{n+1}$, then the MCF of any immersion $F: M \to \bR^{n+1}$ satisfies $0<\cH^{n-1}(S_t) < \infty$ for all $0< t < T_{\mathrm{sing}}$, where $T_{\mathrm{sing}}$ is the first singular time. It also follows that if a smooth closed self-shrinker in $\bR^{n+1}$ has a non-empty self-intersection set $S$, then $0 < \cH^{n-1}(S) < \infty$. See Section \ref{sec:App} for more applications.


\subsection{An intersection principle for weak solutions}

In the second part of the paper, we find that the intersection principle holds for weak solutions of mean curvature flow under certain assumptions, and this is intimately related to the uniqueness of the flow passing through singular times.
The weak solutions we will consider are Brakke flow and level set flow (LSF). See Section~\ref{section:dimension-monotonicity-weak} and~\cite{Ilmanen94} for background.

Controlling the intersection of flows at singular times presents a significant challenge. 
In fact, there are natural examples of flows through singularities which violate the intersection principle, i.e.,\ that violate monotonicity of the dimension of the intersection with smooth flows. 
There exist flows forming isolated conical singularities whose intersection with a smooth flow does not have non-increasing dimension (see Figure~\ref{figure-cone}). 
The failure of intersection dimension monotonicity can occur even for flows which are initially compact and smooth (see Corollary~\ref{cor:bad-ex-1} and Remark~\ref{rmk:smoothBFIntersectionPrinciple}). 
Specifically, in Corollary~\ref{cor:bad-ex-1}, we construct a Brakke flow, associated to a fattening conical singularity, whose intersection with a certain smooth flow does not have a monotone dimension over time. 
The idea is that a level set flow which is fattening will have different ``inner'' and ``outer'' flows through singularities, and one can sometimes expect one of the flows to provide a counterexample to intersection dimension monotonicity. 

\begin{figure}[h]
    \centering
    \captionsetup{type=figure}
    \includegraphics[width=0.45\linewidth]{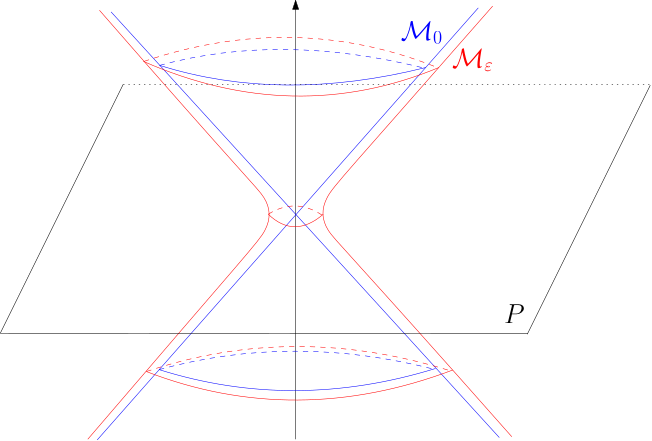}
    \caption{A flow $\cM_t$ with an isolated conical singularity can intersect a plane $P=P_t$ on a $0$-dimensional set at the singular time $t=0$ yet intersect on a higher-dimensional set at a later time $t=\eps$.}
    \label{figure-cone}
\end{figure}

In principle, one may expect that the intersection of weak solutions with smooth flows has non-increasing dimension when the weak solution is unique and has a small singular set. 
This is what we find in Theorem \ref{thm:dim-mono-nonfattening} in the case that the level set flow has finitely many singularities. See Definition \ref{def:lsf-sing} for the meaning of singularities of a potentially fattening LSF. 
The assumption of finiteness of singularities is reasonable since a generic mean curvature flow is conjectured to have only finitely many singularities
(see~\cite[Conjecture~7.1]{CM-MCF} 
and~\cite[Corollary~1.3, Conjecture~1.4]{SWX}). 

\begin{thm}\label{thm:dim-mono-nonfattening}
    Let $M_t$ be a compact level set flow starting from a smooth, closed, embedded hypersurface in $\bb R^{n+1}$. 
    Suppose $M_t$ has finitely many singularities (see Definition~\ref{def:lsf-sing}). Then, the following are equivalent:
\begin{enumerate}
    \item\label{main-item-1} $M_t$ is non-fattening.
    \item\label{main-item-2} The inner and outer flows coincide with $M_t$.
    \item\label{main-item-3} $M_t$ satisfies the intersection principle with respect to smooth MCF. 
    Specifically, if $N_t$ is a smooth closed MCF such that $N_t \not\sbst M_t$ for each $t \in [0, \infty)$, then $t \mapsto \dim\pr{M_t\cap N_t}$ is non-increasing. Moreover, if $\dim(M_{s} \cap N_{s})<n-1$ for some $s$, then $M_t \cap N_t = \emptyset $ for all $t>s$. \label{thmprop3}
\end{enumerate}
\end{thm}

If a level set flow $M_t$, starting from a smooth initial condition, satisfies property~\eqref{thmprop3} from Theorem \ref{thm:dim-mono-nonfattening}, then $M_t$ is non-fattening without any other additional assumptions (see the proof of Theorem \ref{thm:dim-mono-nonfattening} in Section~\ref{sec:non-fat-lsf}). 
In other words, a sufficient condition for an LSF to be non-fattening is that it satisfies the intersection principle with respect to smooth flows. 
It is a well-known problem to characterize when a level set flow is non-fattening, and Theorem \ref{thm:dim-mono-nonfattening} does so when the flow has finitely many singularities. 
Just as Ilmanen used the avoidance principle to characterize level set flow \cite{Ilmanen93, Ilmanen94}, Theorem \ref{thm:dim-mono-nonfattening} suggests that the intersection principle could be used to characterize the fattening of LSFs in general. 
It is an open question to what extent non-fattening of level set flows is equivalent to the intersection principle.

There are two main issues with proving Theorem \ref{thm:dim-mono-nonfattening}.
The first is that when two LSFs intersect each other on a set of codimension greater than two, it is not necessarily true that one of them lies on one side of the other, in contrast with the case when both flows are smooth. 
A typical example is a round shrinking sphere intersecting a shrinking dumbbell at its neck singularity (see Figure~\ref{fig:spheredumbbell}).

One-sidedness of flows with small intersection dimension is crucial for proving monotonicity of the dimension over time. 
To deal with the lack of one-sidedness, we prove a localization result for level set flows with finitely many singularities (Proposition~\ref{prop:localized-LSF} and Remark~\ref{rmk:LSF-localize-nonfattening}). 
This means that we find a natural way to decompose the flow into subsets, such that the union of the LSFs of the subsets is the whole flow. 
This is a very special property, since level set flow is fundamentally non-local. 
For example, if $M$ is a smooth, connected, and closed hypersurface and $M_1$ and $M_2$ are two smooth hypersurfaces with nonempty boundary such that $M = M_1 \cup M_2$, then the level set flow $M_t$ will be a smooth MCF yet $(M_1)_t$ and $(M_2)_t$ instantaneously vanish under LSF~\cite[Theorem~8.1]{EvansSpruck91}.  
Only in special cases does the union of level set flows of subsets give the LSF of the whole set. 

\begin{center}
\captionsetup{type=figure}
    \includegraphics[width=7.5cm]{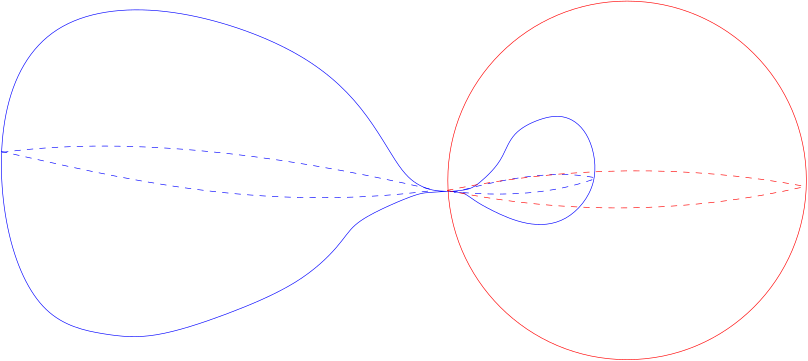}
    \captionof{figure}{
    One-sidedness may fail for level set flows intersecting on a small set.
    In this case, one of the LSFs may be localizable, meaning that it decomposes as a flow of subsets. This figure shows an LSF that could evolve as the union of the flow starting from the portion inside the sphere and another starting from the portion outside the sphere.}\label{fig:spheredumbbell}
\end{center}

An important tool for our proof of Theorem \ref{thm:dim-mono-nonfattening} is White's theorem~\cite{White95} on topological changes of a level set flow through a singular time. 
Intuitively, one expects that White's theorem rules out the potential bad behavior illustrated in Figure~\ref{figure-cone} when the flow is non-fattening. 
In practice, we use White's result to prove the localization result, Proposition~\ref{prop:localized-LSF}, for an LSF with finitely many singularities.

The second issue with proving Theorem~\ref{thm:dim-mono-nonfattening} is the conjecture that non-fattening level set flows must coincide with both their inner and outer flows (see~\cite[Conjecture~2.4]{HershkovitsWhite17}). 
This is an important conjecture which, if true, would confirm that nonfattening LSFs are unique in a strong sense. 
As part of Theorem \ref{thm:dim-mono-nonfattening}, we prove this conjecture in the case of finitely many singularities. 
Hershkovits--White showed that this conjecture is true for flows with mean convex neighborhoods of singularities~\cite{HershkovitsWhite17}, and Bamler--Kleiner proved it in general for MCFs in $\bR^3$~\cite[Theorem 1.8]{BK23}. 
Agreement between the inner and outer flows is used in our proof of Theorem \ref{thm:dim-mono-nonfattening} in order to understand the singularities of the level set flow via an associated Brakke flow.

There are more general conditions than what is stated in Theorem~\ref{thm:dim-mono-nonfattening} which guarantee that an LSF satisfies the intersection principle. 
We define a general class of ``localizable'' level set flows, which loosely means that the flow has no singularities which are locally disconnected at a singular time yet which subsequently flow to become locally connected (see Definition \ref{definition-localizable}). 
For example, a one-sheeted flow desingularizing a two-sheeted cone, as in Figure~\ref{figure-cone}, would not be localizable. 
Our main result, given in Theorem \ref{thm:loc-LSF-mono}, is roughly stated as follows:
\begin{quote}
    \textit{A non-fattening, localizable level set flow with no higher multiplicity planar tangent flows satisfies the intersection principle.}
\end{quote}
The proof of Theorem \ref{thm:loc-LSF-mono} uses an intersection dimension monotonicity result for the class of ``localizable'' Brakke flows (see Definition \ref{definition-splittable} and Theorem \ref{theorem-splittable-implies-GAP}). 

Non-localizable flows, such as one-sheeted flows desingularizing a two-sheeted cone, are a primary source of examples for fattening level set flows.
Thus, localizability is closed related to fattening, and hence the intersection principle.
Although localizability does not necessarily imply non-fattening, under reasonable assumptions, a level set flow is localizable if and only if both the inner and outer flows satisfy the intersection principle with respect to smooth closed MCFs.
See Theorem \ref{thm:loc-LSF-mono} and Remark~\ref{remark:ConverseToLocLSF-IntPrinciple}.
It is an interesting open problem what restrictions on singularities ensure that a level set flow is localizable. 
It is plausible that flows with only cylindrical or spherical singularities would have this property, and so we expect that generic MCF would satisfy the intersection principle, i.e.\ item~\eqref{main-item-3} from Theorem \ref{thm:dim-mono-nonfattening}.

Thus far, our results for weak solutions only consider intersections with smooth flows. 
One could also ask when the intersection of two weak solutions has non-increasing dimension over time. 
One major issue is that two singular flows could switch ``sides'' through their singular set, and the intersection of the singular sets of two flows could lead to unexpected changes in the dimension of the intersection as each flow desingularizes differently.
A potential approach to dealing with this would be to impose a condition akin to Wickramasekera's Hypothesis $\mathrm K$ at the singular times and to prove intersection dimension monotonicity with this assumption~\cite{Wic14-max}. 
We expect that intersection dimension monotonicity would hold for pairs of LSFs with finitely many singularities, at least in low dimensions.

One could generalize the results of this paper in other ways. 
Some nonlinear geometric flows of hypersurfaces are known to satisfy the avoidance principle~\cite[Theorem~5]{ALM13}. 
By analyzing their graphical evolution equations as in Section \ref{sec:Preliminary}, the results of this paper could hold for these flows. 
Also, all our results have been stated only for MCF in $\bR^{n+1}$, but we expect similar statements to hold in ambient Riemannian manifolds with uniformly bounded geometry and a uniformly lower bounded injectivity radius. 
Indeed, the classical avoidance principle can be generalized to weak solutions of MCF in ambient Riemannian manifolds with this controlled geometry~\cite{Ilmanen93, HW23, WhiteAvoidance24}.

We organize the paper as follows.
In Section~\ref{sec:Preliminary}, we prove some preliminary results, including a local finiteness of the $\cH^{n-1}$-measure of intersecting flows and some criteria for one-sidedness. 
In Section~\ref{section:dimension-monotonicity}, we prove the main results for smooth flows, i.e.\ Theorems \ref{theorem main} and \ref{thm:self-int-main}. 
In Section~\ref{section:dimension-monotonicity-weak}, we prove the main results for weak solutions, i.e.\ Theorem \ref{thm:dim-mono-nonfattening} and its generalizations.  
Finally, in Section~\ref{sec:ApplicationsandExamples}, we provide applications of our results and give examples and counterexamples regarding the behavior of intersections of MCFs.

\noindent \textbf{Acknowledgments.} The authors would like to thank Yiqi Huang for helpful conversations,
and would like to thank Jacob Bernstein, Mark Haskins, Sasha Logunov, Bill Minicozzi, Nata\v{s}a \v{S}e\v{s}um, and Lu Wang for insightful comments and discussions. 
We would especially like to thank Andrew Sageman-Furnas for suggesting the application to self-intersections of immersed flow and thank Ilyas Khan for suggesting a sharpening of our local estimate in Section \ref{section:local-measure}.
 
Lee was partially supported by NSF Grant DMS 2304684.
Payne would like to thank the Simons Foundation for its support under the Simons Collaboration on Special Holonomy in Geometry, Analysis and Physics, grant \#488620. 
This material is also based upon work supported by the National Science Foundation under Grant No.\ DMS-1928930, while Payne was in
residence at the Simons Laufer Mathematical Sciences Institute (formerly MSRI) during the Fall 2024 semester.


\section{Preliminary Results}
\label{sec:Preliminary}

In this paper, we let $n \geq 1$ be a positive integer, let $B_r(x)$ denote the ball of radius $r$ centered at a point $x \in \bR^{n+1}$ (or sometimes in $\bR^{n}$, as in Theorem \ref{thm:HJ-bound}), and let $B_r:= B_r(0)$ be the $r$-ball centered at the origin.  

\subsection{Local measure bound for the intersection of flows}
\label{section:local-measure}

We begin with one of the main ingredients to prove Theorem~\ref{theorem main}, a Hausdorff measure estimate on nodal sets of solutions to linear parabolic PDEs established by Huang--Jiang~\cite{HJ}.

\begin{thm}[\!\mbox{\cite[Theorem~1.4]{HJ}}]\label{thm:HJ-bound}
    Let $a^{ij}, b^i,$ and $c$ be real-valued functions on $B_2\times (-4,0]\sbst \bR^n\times \bR$, and suppose there exists $C > 0$ such that for all $(x,t)\in B_2\times (-4,0],$ 
    \begin{align*}
        (1+C)^{-1}[\de^{ij}] \leq [a^{ij}(x,t)] 
        \leq (1+C) [\de^{ij}]
    \end{align*}
    as $n \times n$ matrices, and for each $i,j$ and for all $(x,t),(y,s)\in B_2\times (-4,0],$ 
    \begin{align*}
        |a^{ij}(x,t) - a^{ij}(y,s)| 
        \leq C \sqrt{|x-y|^2 + |t-s|},\quad 
        |b^i(x,t)|+|c(x,t)|\le C.
    \end{align*}
If $w\colon B_2\times (-4,0]\to\bb R$ is a solution to the equation
    \begin{align*}
        \frac{\dd w}{\dd t} 
        = \sum_{i,j=1}^n \frac{\dd}{\dd x_i}\pr{a^{ij}(x,t) \frac{\dd w}{\dd x_j}}
        + \sum_{j=1}^n b^j(x,t) \frac{\dd w}{\dd x_j}
        + c(x,t) w
    \end{align*}
    with $w(\cdot,t_0)$ not identically zero for some $t_0\in (-4,0],$ then 
    \begin{align*}
        \mathcal H^{n-1}\pr{\set{
        x\in B_1: w(x,t_0)=0
        }}
        \le D\pr{n, C, \Lambda}
        <\infty,
    \end{align*}
    where $\Lambda = \int_{B_2\times (-4,0]} w^2 dx dt/ \int_{B_{3/2}\times \set{t_0}} w^2 dx.$
\end{thm}

To apply Theorem~\ref{thm:HJ-bound}, we must first analyze the local graphical behavior of mean curvature flows.
Let $T>0$. 
If $u: \bR^n \times [0, T) \to \bR$ is a $1$-parameter family of smooth functions, the graphs of $u(\cdot, t)$ are a mean curvature flow for $t \in [0, T)$ if the following holds:
\begin{align}\label{graphical-MCF}
\frac{\dd u}{\dd t} = \De u - \frac{\mathrm{Hess}_u(\na u, \na u)}{1+|\na u|^2}.
\end{align}

In the following proposition, we show that the difference of two solutions to~\eqref{graphical-MCF}, with a gradient bound, satisfies a linear parabolic PDE. 
This seems to be a folklore result in mean curvature flow, which is analogous to a well-known result for minimal surfaces~\cite[Lemma 1.26]{ColdingMinicozziMinimalBook} and is a higher-dimensional version of what Angenent used in his application of Sturmian theory to curve shortening flow~\cite{Angenent91}. 

\begin{prop}\label{proposition difference of graphical MCF}
Let $U\subseteq \bR^n$ be an open set, and let $u,v\colon U\times[0,T] \to \bR$ be two smooth solutions of graphical mean curvature flow \eqref{graphical-MCF}.
Define $w:=v-u$. 
Then, for $i,j = 1, \dots, n$, there exist smooth functions $a^{ij}, b^j: U \times[0, T] \to \bR$ such that 
\begin{align}\label{equation difference of graphical MCF}
       \frac{\dd w}{\dd t} 
       = \sum_{i,j=1}^n \frac{\dd}{\dd x_i}\pr{a^{ij}(x,t) \frac{\dd w}{\dd x_j}} 
       + \sum_{j=1}^n b^j(x,t) \frac{\dd w}{\dd x_j}. 
\end{align}
Moreover, if $U$ is a bounded set such that $\sup_{U \times [0,T]} |\na u| + |\na v| < \infty$ and the flows have bounded curvature in $U\times [0,T]$, then for each compact set $K\subseteq U$, there exists $C< \infty$ such that the following properties hold:
\begin{enumerate}
    \item\label{equations for PDE coefficients} For each $(x,t) \in K \times [0,T]$, the coefficient matrix $[a^{ij}(x,t)]$ satisfies
\begin{align*}
   (1+C)^{-1}[\de^{ij}] \leq [a^{ij}(x,t)] \leq [\de^{ij}].
\end{align*}
    \item\label{equation aij parabolic Lipschitz} For each $i,j = 1, \dots, n,$ $a^{ij}$ is Lipschitz in the parabolic distance on $K \times [0, T]$:
\begin{align*}
|a^{ij}(x,t) - a^{ij}(y,s)| \leq C \sqrt{|x-y|^2 + |t-s|}
\end{align*}
for all $(x,t), (y,s)\in K \times [0,T]$.
    \item\label{equation bj bound} For each $j=1,\cdots,n,$ $\sup_{K \times [0, T]}|b^j(x,t)|\leq C$.
\end{enumerate}
\end{prop}

\begin{proof}
Define $F:\bR^{n^2} \times \bR^n\to\bb R$
by
\begin{align*}
F(P, q) := \sum_{i=1}^n P_{ii} - \sum_{i,j=1}^n \frac{P_{ij}q^i q^j}{1+|q|^2},
\end{align*}
where $P=[P_{ij}]\in \bR^{n^2}$ is interpreted as an $n \times n$ matrix and $q =(q^j)\in \bR^n$ is interpreted as a vector.
We may rewrite the graphical MCF equation \eqref{graphical-MCF} as 
\begin{align*}
    \frac{\dd u}{\dd t} = F\pr{ \mathrm{Hess}_{u} (x,t), \na u (x,t)}.
\end{align*}
Now, for $\te \in [0, 1]$, define $w^{\te}:= \te v + (1-\te)u$. 
We derive
\begin{align}\label{equation dw/dt}
    \frac{\dd w}{\dd t} &= \frac{\dd v}{\dd t} - \frac{\dd u}{\dd t}\nonumber\\
    &= F\pr{\mathrm{Hess}_{v} (x,t), \na v (x,t) }
    - F\pr{\mathrm{Hess}_{u} (x,t), \na u (x,t)}\nonumber\\
    &= \int_{\te =0}^1 \frac{d}{d\te} F\pr{\mathrm{Hess}_{w_{\te}} (x,t), \na w_{\te} (x,t)} d\te\nonumber\\
    &= \sum_{i,j=1}^n \Bigg(\int_{\te =0}^1 
    \frac{\dd F}{\dd P_{ij}} \pr{\mathrm{Hess}_{w_{\te}} (x,t), \na w_{\te} (x,t)} d\te\Bigg)\frac{\dd^2 w}{\dd x_i \dd x_j}\nonumber\\
    & \quad 
    + \sum_{k=1}^n \Bigg(\int_{\te=0}^1 \frac{\dd F}{\dd q^j} \pr{\mathrm{Hess}_{w_{\te}} (x,t), \na w_{\te} (x,t)}d\te \Bigg) \frac{\dd w}{\dd x_j}.
\end{align}
For $i, j = 1, \dots, n$, define the following coefficients:
\begin{align*}
    a^{ij}(x,t) &:= \int_{\te =0}^1 \frac{\dd F}{\dd P_{ij}} \pr{\mathrm{Hess}_{w_{\te}} (x,t), \na w_{\te} (x,t)} d\te, \text{ and}\\
    b^j(x,t)  &:= \int_{\te=0}^1 \frac{\dd F}{\dd q^j} \pr{\mathrm{Hess}_{w_{\te}} (x,t), \na w_{\te} (x,t)} d\te\nonumber\\
    & \qquad 
    - \sum_{i=1}^n \frac{\dd}{\dd x_i} \Bigg(\int_{\te =0}^1 \frac{\dd F}{\dd P_{ij}} \pr{\mathrm{Hess}_{w_{\te}} (x,t), \na w_{\te} (x,t)} d\te\Bigg).
\end{align*}
We note that $a^{ij}(x,t)$ and $b^j(x,t)$ are smooth since $u$, $v$, and $F$ are smooth. Combining~\eqref{equation dw/dt} with the definitions of $a^{ij}$ and $b^j$, we find that $w$ satisfies~\eqref{equation difference of graphical MCF}.

We now compute the partial derivatives of $F$:
\begin{align}\label{equation derivatives of F 1}
    \frac{\dd F}{\dd P_{ij}}(P,q) &= \de^{ij} - \frac{q^i q^j}{1+|q|^2}, \text{ and}\\
   \label{equation derivatives of F 2} \frac{\dd F}{\dd q^j}(P,q) &= - \sum_{i=1}^n \frac{(P_{ij}+ P_{ji})q^i}{1+|q|^2} + \sum_{a,b=1}^n \frac{2q^j P_{ab}q^a q^b}{(1+|q|^2)^2}.
\end{align}
Applying~\eqref{equation derivatives of F 1} to $w_{\te}$, we have
\begin{align*}
    \frac{\dd F}{\dd P_{ij}} \pr{\mathrm{Hess}_{w_{\te}} (x,t), \na w_{\te} (x,t)} = \de_{ij} - \frac{1}{1+|\na w_{\te}|^2} \frac{\dd w_{\te}}{\dd x_i} \frac{\dd w_{\te}}{\dd x_j}.
\end{align*}

Let us now consider a bounded set $U \subseteq  \bR^n$ such that $\sup_{U \times [0,T]} \pr{|\na u| + |\na v|} < \infty$. 
This implies that there is a constant $C< \infty$ such that for each $\te \in [0,1]$, 
\begin{align}\label{equation na bound wte}
\sup_{U \times [0,T]} |\na w_{\te}| \leq C.
\end{align}
Now, we note that the matrix $[a^{ij}(x,t)]$ is symmetric. 
If we let $Y \in \bR^n$ be a unit vector with entries $Y_j$, then for each $(x,t) \in U \times [0, T]$, 
\begin{align*}
    \sum_{i,j=1}^n a^{ij}(x,t) Y_i Y_j 
    &=  \int_{\te=0}^1 
    \sum_{i,j=1}^n \pr{\de_{ij} - \frac{1}{1+|\na w_{\te}|^2} \frac{\dd w_{\te}}{\dd x_i} \frac{\dd w_{\te}}{\dd x_j}} (x,t)\cdot Y_{i}Y_j\, d\te\\
    &= \int_{\te=0}^1 \pr{1 - \frac{\pair{\na w_{\te}, Y}^2}{1+|\na w_{\te}|^2}}(x,t) \, d\te\\
    &\geq \int_{\te=0}^1 \pr{1 - \frac{|\na w_{\te}|^2}{1+|\na w_{\te}|^2}}(x,t) \, d\te\\
\intertext{\small Using~\eqref{equation na bound wte}, we can find a relabelled constant $C< \infty$ such that}
    &\geq (1+C)^{-1} .
\end{align*}
This lower bounds the Rayleigh quotient of the symmetric matrix $[a^{ij}(x,t)]$ by a positive number independent of $(x,t)$, which proves that there is a uniform positive lower bound for the smallest eigenvalue of $[a^{ij}(x,t)]$. 
This proves the lower bound of~\eqref{equations for PDE coefficients}. 
The expression above also allows us to upper bound the Rayleigh quotient by $1$, which gives the upper bound in~\eqref{equations for PDE coefficients}.

Since $u$ is a smooth graphical mean curvature flow with bounded curvature in $U \times [0, T]$, Ecker--Huisken's local curvature and higher order derivative estimates~\cite[Theorem 3.4]{EckerHuisken91} imply that for a fixed compact set $K\subseteq  U$, given an integer $k>0,$ there exists $C(k)< \infty$ such that 
\begin{align*}
\sup_{K \times [0,T]} |A|+|\na^k A| \leq C(k),
\end{align*}
where $A$ is the second fundamental form of the hypersurface $\mathrm{graph}\!\pr{u(\cdot, t)|_{K}}$ at time $t \in [0,T]$. 
Using the fact that bounds on $|\na u|$ combined with bounds on $|\na^{k-2} A|$ give bounds on $|\na^k u|$~\cite[Lemmas 2.1 and 2.2]{Cooper}, we find that for each $k\geq 0$,
\begin{align*}
\sup_{K \times [0,T]} |\na^k u| \leq C(k).
\end{align*}
Similar reasoning applies to $v$, so we find that for each $\te \in [0,1]$ and each $k \geq 1$, there is a relabeled constant $C(k)$ such that
\begin{align}\label{equation k derivatives of F wte}
    \sup_{K \times [0,T]}|\na^k w_{\te}| \leq C(k).
\end{align}
Combining ~\eqref{equation derivatives of F 1} and ~\eqref{equation derivatives of F 2} with the definition of $b^j$, we may bound $b^j$ using~\eqref{equation na bound wte} and~\eqref{equation k derivatives of F wte}. This implies that there is $C< \infty$ such that 
\begin{align*}\sup_{K \times [0,T]} |b^j(x,t)| \leq C,\end{align*}
which justifies~\eqref{equation bj bound}. 
Moreover, using~\eqref{equation k derivatives of F wte}, we find that $a^{ij}(x,t)$ is smooth in $K \times [0,T]$. 
In particular, it is Lipschitz in the parabolic distance, which justifies~\eqref{equation aij parabolic Lipschitz}.
\end{proof}

We will now use Theorem~\ref{thm:HJ-bound} and Propositions~\ref{proposition difference of graphical MCF} to find a local measure bound on the intersection of mean curvature flows.
This will ultimately apply to the regular part of Brakke flows under appropriate assumptions.

\begin{thm}\label{thm-local-bound}
	Let $U\subseteq  \bR^{n+1}$ be an open set, and suppose that $M_t$ and $N_t$ are smooth, connected $n$-dimensional mean curvature flows which exist and are properly embedded in $U$ for $t \in [0,T]$. Then, for each $t\in (0,T)$, one of the following holds:
 \begin{itemize}
     \item $M_t = N_t$, or
     \item for each compact set $K\subseteq  U$,
     $\cH^{n-1}(M_t \cap N_t \cap K) < \infty.$
 \end{itemize}
\end{thm}

\begin{rmk}
The upper bound for $\cH^{n-1}(M_t \cap N_t \cap K)$ depends on the geometry of $M_t \cap K$ and $N_t \cap K$, as well as an upper bound on $\La$ from Theorem \ref{thm:HJ-bound}, suitably adapted to our setup. 
\end{rmk}

\begin{proof}	
By standard parabolic theory, a mean curvature flow in $U$ is spatially real analytic after the initial time (see~\cite[Remark 1.5.3]{Mantegazza}). 
This means that $M_t$ and $N_t$ are each real analytic in $U$ for $t>0$. 

Take $B_r(x_0) \subseteq U$ for some $r>0$ and $x_0 \in U$. 
If $M_t \cap B_r(x_0) = N_t \cap B_r(x_0)$, then the identity theorem for real analytic hypersurfaces would imply that $M_t = N_t$ in $U$ for $t>0$ (see~\cite{mityagin2020zero}). 
For the purposes of this theorem, we may assume that $M_t \neq N_t$ for each $t>0$. 
Then, for each $r>0,$ $t>0,$ and $x_0 \in U$ such that $B_r(x_0) \subseteq U$, we have that $M_t \cap B_r(x_0) \neq N_t \cap B_r(x_0)$. 

Let $K \subseteq U$ be a compact subset of $U$. 
We will now prove that for some $\eps>0$, 
\begin{align}\label{local-int-bound}
    \cH^{n-1}(M_t \cap N_t \cap K) < \infty \text{ for }t \in (0,\eps).
\end{align} 
By applying the same argument to $M_{t-t_0}$ and $N_{t- t_0}$, where we have translated $t_0 \in (0,T)$ to time $0$, we will conclude the theorem. 
We may suppose that $M_0\cap N_0 \cap K \neq \emptyset$. 
Indeed, if $M_0 \cap N_0 \cap K = \emptyset$, then there exists some $\eps>0$ such that $M_t \cap N_t \cap K= \emptyset$ for $t \in (0,\eps)$ and the theorem follows.

Since $M_0 \cap N_0 \cap K \neq \emptyset$, we may assume that $0 \in M_0 \cap N_0\cap K$ after a spatial translation. 
Near $0,$ choose unit normal vectors of $M_0$ and $N_0$ to be $\N_{M_0}$ and $\N_{N_0}$ such that $\pair{\N_{M_0}(0), \N_{N_0}(0)}\ge 0.$
Then, we can find a unit vector $\N \in \bR^{n+1}$ such that 
$\pair{\N_{M_0}(0), \N}$ and $\pair{\N_{M_0}(0), \N}$ are positive. 
Hence, by the smoothness of $M_0$ and $N_0$, there is $r>0$ such that $M_0 \cap B_r$ and $N_0 \cap B_r$ can both be written as graphs over the plane $P:=\N^\perp\sbst\bb R^{n+1}$.
We know that $M_t\cap B_r$ and $N_t\cap B_r$ have uniformly bounded second fundamental form for $t\in [0, \varepsilon].$
Thus, by possibly shrinking $r>0$, we have that $M_t\cap B_r$ and $N_t\cap B_r$ are graphs over $P.$
To be more specific, we find $u,v\colon P\cap B_{r}\to\bb R$ with $|\n u| + |\n v|<\infty$ such that
\begin{itemize}
	\item the connected component of
	$M_0 \cap B_{r}$ containing $0$ is the graph of $u$ over $P \cap B_{r}$, and 
	\item the connected component of $N_0 \cap B_{r}$ containing $0$ is the graph of $v$ over $P \cap B_{r}$.
\end{itemize}
By the bounds on the second fundamental form, $M_t \cap B_{r/2}$ and $N_t \cap B_{r/2}$ correspond to $u$ and $v$ evolving by the graphical mean curvature flow equation ~\eqref{graphical-MCF}.  
The Ecker--Huisken interior estimate~\cite[Theorem~2.1]{EckerHuisken91} implies $|\n u(x,t)|+ |\n v(x,t)|<\infty$
for $(x,t)\in \pr{P\cap B_{r/10}}\times [0, r^2/100n]$.
Since $M_t \cap B_r \neq N_t \cap B_r$ for each $r>0$ and each $t>0$, we have that the function
$$w\colon \pr{P \cap B_{r/10}}\times [0, r^2/100n] \to \bR$$
defined by $w:= v-u$ is not identically zero. 
Based on Proposition~\ref{proposition difference of graphical MCF}, $w$ satisfies the assumptions of Theorem~\ref{thm:HJ-bound}, so we get that 
\begin{align}\label{H-finite1}
\mathcal H^{n-1}\pr{\{w(\cdot, t)=0\} \cap B_{r/10}} < \infty 
\end{align}
for $t\in (0, r^2/100n].$

Since $M_t$ and $N_t$ are graphs of $u(\cdot, t)$ and $v(\cdot, t)$ over $P \cap B_{r/10}$, the set $\{w(\cdot, t)=0\}$ corresponds to $M_t \cap N_t \cap B_{r/10}$. 
Specifically, we have that the connected component of $M_t \cap N_t \cap B_{r/10}$, which at $t=0$ contains $0$, is precisely $\mathrm{graph}\big(u|_{\{w(\cdot, t)=0\}}\big) = \mathrm{graph}\big(v|_{\{w(\cdot, t)=0\}}\big)$. 
Using the gradient bounds on $u(\cdot, t)$ and $v(\cdot, t)$ in $P \cap B_{r/10}$ for $t \in [0, r^2/100n]$, we have that $u(\cdot, t)$ and $v(\cdot, t)$ have a uniform Lipschitz bound for $t \in [0, r^2/100n]$. 
This implies that the graph map
 \begin{align}\label{equation graph map}
        z\in P\cap B_r
        \mapsto
        z + u(z) \N
        \in {\rm graph}~u.
    \end{align}
is Lipschitz (and likewise for $v$). Since $M_t \cap N_t \cap B_{r/10}$ is the image of $\{w(\cdot, t)=0\}\subseteq P \cap B_{r/10}$ by the graph map~\eqref{equation graph map}, which is Lipschitz for $t \in [0, r^2/100n]$, \eqref{H-finite1} implies
\begin{align}\label{H-finite}
\mathcal H^{n-1}\pr{M_t \cap N_t \cap B_{r/10}} < \infty. 
\end{align}

	We may now use the local estimate \eqref{H-finite}, which holds near any point in $M_0\cap N_0 \cap K,$ to conclude the proof.
    By \eqref{H-finite}, we know that for any $x\in M_0\cap N_0\cap K,$ there exists $r_x>0$ such that 
    \begin{align}\label{H-finite-x}
	\mathcal H^{n-1}\pr{M_t\cap N_t\cap B_{r_x}}
	< \infty
	\end{align}
    for all $t\in \pr{0, r_x^2}.$
    Since $K$ is a compact set, by the proper embeddedness assumption, $M_0
    \cap N_0\cap K$ is also a compact set.
    Thus, finding a finite cover of $M_0\cap N_0\cap K$ by $B_{r_{x_1}}(x_1),\dots, B_{r_{x_\ell}}\pr{x_\ell}$ for some $x_1,\dots, x_\ell\in M_0\cap N_0\cap K,$
    we can apply~\eqref{H-finite-x} to each $x_i$ and conclude
	\begin{align*}
	\mathcal H^{n-1}\pr{M_t\cap N_t\cap K}
	< \infty
	\end{align*}
    for $t\in \pr{0, r^2}$ where $r:=\min\set{r_{x_1},\cdots, r_{x_\ell}}>0.$
    This finishes the proof of \eqref{local-int-bound}.
    Applying the same argument to $M_{t-t_0}$ and $N_{t-t_0}$ for $t_0 \in (0,T)$, we conclude the theorem.
\end{proof}


\subsection{One-sidedness property}\label{sub:oneside}
We collect a few one-sidedness criteria that will be used in the proofs of the main theorems. 
First, we mention a lemma proven in Huang--Jiang~\cite{HJ}.

\begin{lem}[\!\mbox{\cite[Lemma~7.5]{HJ}}]\label{lemma HJ one side} 
    Let $B_1 \subseteq  \bR^n$, and let $u: B_1 \to \bR$ be a continuous function. 
    If $\{u>0\} \neq \emptyset$ and $\{u<0\} \neq \emptyset$, then 
    $\dim(\{u=0\}) \geq n-1.$
    In fact, $\cH^{n-1}\pr{\set{u=0}}>0.$
\end{lem}

In \cite{HJ}, Huang--Jiang only stated the conclusion about dimension in their Lemma~7.5.
However, based on a direct measure-theoretic argument, they were able to get the conclusion about positive $(n-1)$-dimensional Hausdorff measure in the course of their proof.

With a similar argument, we can obtain the following criterion. 
It also applies to a global setting.

\begin{lem}\label{lem:low-dim-one-side}
    Let $U$ be an $n$-dimensional connected smooth hypersurface in a ball $B\sbst \bb R^{n+1}$, and let $K$ be a closed domain in $\bR^{n+1}.$ 
    If $\cH^{n-1}\pr{U\cap \dd K} =0,$ then either $U\sbst K$ or $U\sbst \ovl{B\setminus K}.$
\end{lem}

\begin{proof}
    We will prove a more general statement in the following claim.
    
    \begin{claim}\label{claim:one-side}
        Let $B= B_r(x)\subset\bR^{n+1}$ be an open ball and consider a subset $X \subseteq B$. 
        If $K\subseteq \bR^{n+1}$ is a closed domain such that $X\setminus \dd K$ is path connected,
        then, either $X \subseteq K$ or $X \subseteq \ov{B \setminus K}$.
    \end{claim}

    To prove Claim~\ref{claim:one-side}, suppose for a contradiction that there exist $p,q \in X$ such that $p \in \inte K$ and $q \in B \setminus K$. 
    By assumption, there is a continuous path 
    $\ga: [0,1] \to X \setminus \dd K$
    such that $\ga(0)=p$ and $\ga(1)=q$. 
    
    Note the following disjoint decomposition: $B = \inte K \cup \dd K \cup (B\setminus K )$. 
    For each $s \in [0,1]$, we have that $\ga(s) \notin \dd K,$ which implies
    \begin{align*}\ga(s) \in \inte K\cup (B\setminus K).\end{align*}
    Let $\cI \subseteq \inte K$ be the path connected component of $\inte K$ containing $p$, and let $\cJ\sbst B\setminus K$ be the path connected component of $B\setminus K$ containing $q.$
    Since $\inte K$ and $B\setminus K$ are open, they are locally path connected, so $\cI$ and $\cJ$ are open. 
    Since $\inte K$ and $B\setminus K$ are disjoint,
    it follows that $\cI$ and $\cJ$ are disjoint. 
    Thus, $\cI \cup \cJ$ is not connected and hence not path connected.

    On the other hand, using the path $\ga$ from $p \in  \cI$ to $q \in \cJ$, we conclude that $\cI \cup \cJ$ is path connected. 
    This is a contradiction, which implies that either $X \subseteq K$ or $X \subseteq \ov{B \setminus K}$. 
    This proves Claim~\ref{claim:one-side}.

    Next, we prove the following folklore statement whose proof is analogous to that of Lemma \ref{lemma HJ one side}.

    \begin{claim}\label{claim:path-conn}
        Let $(M, g)$ be a smooth Riemannian $n$-manifold, and let $U\subseteq M$ be a connected open subset. 
        If $S\subseteq U$ is a closed subset and $\cH^{n-1}(S) =0$, then $U\setminus S$ is path connected. 
    \end{claim}

    We first prove the claim when $U$ is an open ball in $\bb R^n.$
    The proof idea is similar to \cite[Lemma~7.5]{HJ}.
    Suppose for a contradiction that $U\setminus S$ is not path connected, and take $x_1,x_2\in U\setminus S$ that are in different components $U_1$ and $U_2$ of $U\setminus S.$
    By the closedness of $S,$ we can find $r>0$ such that $B_r(x_1)\sbst U_1$ and $B_r(x_2)\sbst U_2.$
    Take the hyperplane $H_1$ which is orthogonal to $x_1-x_2$ and passes through $x_1,$ and similarly take the hyperplane $H_2$ which is orthogonal to $x_1-x_2$ and passes through $x_2.$
	For any $x\in H_1\cap B_{r}(x_1),$ there then exists a unique $y\in H_2\cap B_{r}(x_2)$ such that $x-y$ is parallel to $x_1-x_2.$
	Let $\ell_{x}$ be the line segment connecting $x$ and $y.$
    Since $x\in U_1$ and $y\in U_2$ and since $U_1 \cup U_2$ is not path connected, there exists $s_x\in \ell_x\cap S.$
	By this construction, we know that $|s_x - s_{x'}|\ge |x-x'|$ for all $x,x'\in H_1\cap B_{r}(x_1).$
	That is, the inverse of the map $x\in H_1\cap B_{r}(x_1) \mapsto s_x\in \cup_x\set{s_x}\sbst U$ is a $1$-Lipschitz map.
	Thus,
	\begin{align*}
		\mathcal H^{n-1}(S)
		\ge \mathcal H^{n-1}\!\pr{
		\set{s_x: x\in H_1\cap B_{r}(x_1)}}
		\ge \mathcal H^{n-1} \pr{H_1\cap B_{r}(x_1)}
		= c(n)\cdot r^{n-1}>0.
	\end{align*}
    This violates the assumption that $\cH^{n-1}(S) =0.$
    Thus, we prove Claim~\ref{claim:path-conn} when $U$ is a Euclidean ball.

    Next, we show that the general case of Claim~\ref{claim:path-conn} follows from the local Euclidean case done above.
    In general, when $U$ is a connected open set in $M,$ given any $p,q\in U\setminus S,$ we can take a smooth curve $\gamma\colon[0,1]\to U$ in $U$ such that $\gamma(0)=p$ and $\gamma(1)=q.$
    For each $x \in \ga([0,1])$, we choose an open ball $V_x\subseteq M$ such that there is a Lipschitz diffeomorphism $\varphi_x: V_x \to B_{r_x}(0)$.
    By the compactness, we can cover $\gamma\pr{[0,1]}$ by finitely many such balls $V_i:=V_{p_i}$ for some $p_i\in\gamma\pr{[0,1]}$ and $i=1,\cdots,\ell.$
    We may assume $p_1=p$ and $p_\ell=q.$

    Suppose $V_i$'s are arranged such that $p_{i+1}\in V_i$ for all $i=1,\cdots,\ell-1.$
    By the fact that $\varphi_{p_i}$'s are Lipschitz and the local case done above, $V_i\setminus S$ is path connected.
    Thus, we can find a path $\gamma_i\colon [0,1]\to V_i\setminus S$ connecting $p_i$ and $p_{i+1}.$
    Concatenating these $\gamma_i$'s, we obtain a curve in $U\setminus S$ that connects $p$ and $q.$
    This finishes the proof of Claim~\ref{claim:path-conn}.

    Combining Claims~\ref{claim:one-side} and~\ref{claim:path-conn}, we prove the lemma.    
\end{proof}


\section{Dimension Monotonicity of the Intersection of Smooth Flows}
\label{section:dimension-monotonicity}


\subsection{Smoothly embedded flows}

In this section, we will prove Theorem~\ref{theorem main}.
The main ingredients are Theorem~\ref{thm-local-bound} and Lemma~\ref{lem:low-dim-one-side}.

First, we give a backwards uniqueness result for smooth mean curvature flows. 
H. Huang showed that there is backwards uniqueness for complete, smooth mean curvature flows with bounded second fundamental form~\cite{HuangBackwards}. 
Building on this, we prove a backwards uniqueness result under the assumptions of Theorem \ref{theorem main}, where one flow is noncompact and has possibly unbounded second fundamental form. 
One potential obstacle is that it is possible for a smooth mean curvature flow of a noncompact hypersurface to become instantaneously compact. 
An example of a flow with such behavior is the level set flow of a noncompact smooth curve with an appropriately constructed cusp-like end. See Ilmanen's work for examples with this behavior~\cite[7.3]{Ilmanen92}.

By a proper mean curvature flow, we mean that each time-slice is a proper embedding and that the spacetime map defining the flow is continuous and proper. See~\cite[Remark 1.2, Lemma C.1]{PeacheyNonuniqueness} for the necessity of this assumption in the classical avoidance principle when one flow is noncompact.

\begin{lem}\label{lemma-backwards-uniqueness}
Let $M$ and $N$ be complete, connected, smooth, and properly embedded hypersurfaces in $\mathbb{R}^{n+1}$ such that at least one of these hypersurfaces is closed. 
Suppose the proper mean curvature flows $M_t$ and $N_t$ exist and are smooth for $t\in [0, T)$, starting from $M$ and $N$.
If there is $t_0 \in [0, T)$ such that $M_{t_0}= N_{t_0}$, then $M =  N$. 
\end{lem}

\begin{proof}
By a result of H. Huang~\cite{HuangBackwards}, complete smooth mean curvature flows with bounded second fundamental form satisfy backward uniqueness. 
In particular, smooth compact mean curvature flows satisfy backward uniqueness. 
Without loss of generality, we assume $M_t$ is a compact flow.

If $M_{t_0} = N_{t_0}$ for some $t_0 \in [0,T)$, then $N_{t_0}$ is compact.    
Since $N_t$ is a smooth proper mean curvature flow, $N_t$ converges to $N_{t_0}$ locally smoothly on compact subsets of $\bR^{n+1}$ as $t \!\nnearrow \! t_0$. 
Combined with the proper embeddedness of $N_t$, this implies that there is $\eps>0$ such that $N_t$ is compact for $t \in [t_0-\eps, t_0]$. Now, define 
\begin{align*}\eps^* := \sup\{\eps>0\,:\,N_t \text{ is compact for each }t \in [t_0 - \eps, t_0]\}.\end{align*}
By construction, $N_t$ is compact and smooth for $t \in (t_0 - \eps^*, t_0]$ and $\eps^*>0$. 
Since $M_{t_0} = N_{t_0}$ and both flows are compact for $t \in (t_0 - \eps^*, t_0]$, Huang's backward uniqueness theorem~\cite{HuangBackwards} says that 
\begin{align}\label{equation huang backward}
    M_t = N_t
    \text{ for }
    t \in (t_0-\eps^*, t_0].
\end{align}
Since $M_t$ and $N_t$ are both smooth and $M_t$ is compact for each $t \in [0, T)$, we have that $M_t =N_t$ converges smoothly to $M_{t_0 - \eps^*}$ as $t \ssearrow t_0 - \eps^*$, and so $M_{t_0 - \eps^*} = N_{t_0 - \eps^*}$. 

If $t_0 - \eps^* > 0$, then as before, smooth convergence implies that there is $\eps>0$ such that $N_{t}$ is compact for $t \in [t_0 - \eps^* - \eps, t_0]$. This contradicts the definition of $\eps^*$, so we find that $t_0 - \eps^* = 0$. By~\eqref{equation huang backward} and smooth convergence of the flow as $t \ssearrow 0$, we have that $M= M_0 = N_0 = N$.
\end{proof}

\begin{proof}[Proof of Theorem~\ref{theorem main}]
Without loss of generality, we may assume that $M$ is closed. 
If $M \cap N = \emptyset$, then the result follows from the classical avoidance principle~\cite[Theorem~2.2.1]{Mantegazza}, so we may additionally assume without loss of generality that $M \cap N \neq \emptyset$.

Now, we note that $M_t \neq N_t$ for each $t \in [0, T)$ based on the assumption $M\neq N$ and Lemma~\ref{lemma-backwards-uniqueness}.    
For $t\in [0, T),$ we let $Z_t:= M_t\cap N_t.$ 
Since $M_t \neq N_t$, Theorem~\ref{thm-local-bound} and the compactness of $M_t$ implies that 
 \begin{align}\label{finite-H-bound}
	\mathcal H^{n-1}\pr{Z_t}
	< \infty
	\end{align}
for all $t\in (0, T).$ 
This concludes the first part of Theorem \ref{theorem main}. 

Next, we will prove that $t \mapsto \dim Z_t$ is non-increasing. 	
We will prove the following claim, which combined with~\eqref{finite-H-bound} will allow us to finish the proof of the theorem.

\begin{claim}
\label{claim:Zt0-empty}
    If $\cH^{n-1}\pr{Z_{t_0}}=0$ for some $t_0 \in [0,T)$, then $Z_t = \emptyset$ for all $t \in (t_0, T)$.
\end{claim}

\begin{proof}[Proof of Claim~\ref{claim:Zt0-empty}]
    It is enough to prove this claim for $t_0 = 0$, so we suppose without loss of generality that $\cH^{n-1}\pr{Z_0}=0$. 
    Therefore, by Lemma~\ref{lem:low-dim-one-side}, $M$ and $N$ lie on one side of each other. 
    A well-known extension of the classical avoidance principle then says that the flows $M_t$ and $N_t$ lie on one side of each other~\cite[Corollary~2.2.3 and Remark~2.2.4]{Mantegazza}.\footnote{This fact also follows from viewing the flows as weak set flows (of domains and hypersurfaces) and appealing to Lemma \ref{lem:Ilm-contain}.} 
    This follows by slightly perturbing one of the flows so that $M$ and $N$ are initially disjoint and then using continuous dependence of the flow on the initial condition. 
    In particular, we find that there is a smooth compact domain $K_t$ such that $\dd K_t = M_t$ and such that either $N \subseteq K_t$ or $N \subseteq \ov{\bR^{n+1}\setminus K_t}$ for all $t\in[0,T).$
    
    We will now see that $M_t$ and $N_t$ lie strictly on one side of each other, i.e., $M_t \cap N_t = \emptyset,$ for $t \in (0,T)$. 
    This is well-known to experts,\footnote{See \cite[Theorem~8.2]{EvansSpruck91} and \cite[\S 4]{EvansSpruck92} for the case when both hypersurfaces are compact.
    The fact also follows from~\cite[Proposition~3.3]{CHHW}.
    } 
    but we provide some details here. 
    As in the proof of Theorem~\ref{thm-local-bound}, we get that for each $x \in Z_0$, $M_t \cap B_r(x)$ and $N_t \cap B_r(x)$ correspond to $u(\cdot, t), v(\cdot, t): P \cap B_r(x) \to \bR$ and evolve by the graphical mean curvature flow equation~\eqref{graphical-MCF}.
    The Ecker--Huisken interior estimate~\cite[Theorem~2.1]{EckerHuisken91} says $|\n u(y,t)|+ |\n v(y,t)|<\infty$
    for $(y,t)\in \pr{P\cap \ovl B_{r_x/10}(x)} \times [0, r_x^2/100n].$
    The one-sidedness property then implies 
    \begin{align*}
        u(y, t)\ge v(y, t)
        \text{ for all }(y, t)\in \pr{P\cap \ovl B_{r_x/10}(x)} \times \left[0, \frac{r_x^2}{100n} \right]
    \end{align*}
    after swapping $u$ and $v$ if necessary.
    By the strong maximum principle for parabolic PDEs~\cite[Theorem~2.7]{Lieberman},
    \begin{align*}
        u(y, t)\neq v(y, t)
        \text{ for all }(y, t)\in \pr{P\cap B_{r_x/10}(x)} \times \left(0, \frac{r_x^2}{100n} \right]
    \end{align*}
    since we know $M_t\neq N_t.$
    This again works for all $x.$ 
    Finding a finite cover of the compact set $M\cap N$ by $B_{r_{1}}(x_1),\dots, B_{r_\ell}\pr{x_\ell}$ for some $x_1,\dots, x_\ell\in M\cap N,$ we conclude that $M_t\cap N_t=\emptyset$ for $t\in \pr{0,\ovl r^2}$ where $\ovl r := \min\set{r_1,\cdots,r_\ell} > 0.$
    The standard avoidance principle then implies $M_t\cap N_t=\emptyset$ for $t\in (0, T),$ and hence we get
    $\dim Z_t=0\le \dim Z_0.$ This completes the proof of the claim.
\end{proof}

We go back to the proof of Theorem~\ref{theorem main}.
We may now consider two cases for $Z_0$: $\dim Z_0 < n-1$ or $\dim Z_0 \geq n-1$.

If $\dim Z_0 < n-1$, then $Z_t = \emptyset$ for all $t \in (0,T)$ by Claim~\ref{claim:Zt0-empty}. 
This implies that $\dim Z_t = 0 \leq \dim Z_0$ and so $t \mapsto \dim Z_t = \dim \pr{M_t \cap N_t}$ is non-increasing in this case.

If $\dim Z_0\ge n-1,$ then \eqref{finite-H-bound} implies $\dim Z_t \le n-1\le \dim Z_0$ for all $t \in (0,T)$. Define 
\begin{align}\label{t0-embedded-flow}
    t_0 &:= \inf \set{t\in[0,T): \dim Z_t < n-1} \in [0,T).
\end{align}
If $t_0$ is an infimum over an empty set, i.e.\ if $\dim Z_t \geq n-1$ for all $t \in [0,T)$, then we set $t_0:= T$. In this case, since $\dim Z_0 \geq n-1$ and $\dim Z_t = n-1$ for $t \in (0,T)$, we have that $t \mapsto \dim Z_t$ is non-increasing.

Now, suppose that $t_0 \in [0,T)$. Then, by construction, we may find a sequence $t_i \ssearrow t_0$ such that $\dim Z_{t_i}< n-1$ for each $i$. 
Applying Claim~\ref{claim:Zt0-empty} to each $t_i$, we find that $Z_t = \emptyset$ for each $t \in (t_0, T)$. 
Thus, we have that $\dim Z_t = n-1$ for $t \in (0,t_0)$. 
We note that it is possible that $t_0 = 0$ and the interval $(0,t_0)$ is empty. 
Thus, for $t \in (0,t_0)$, we have that $\dim Z_0 \geq n-1 = \dim Z_t$, and for $t \in (t_0, T)$, we have that $n-1 \geq \dim Z_{t_0} \geq \dim Z_t=0$. 
Thus, we find that $t \mapsto \dim Z_t = \dim \pr{M_t \cap N_t}$ is non-increasing. 

This concludes the proof of the dimension monotonicity result in Theorem \ref{theorem main}. 
From~\eqref{finite-H-bound} and our argument for the dimension monotonicity result, we have that $0< \cH^{n-1}(Z_t)< \infty$ for $t \in (0,t_0)$. By Claim \ref{claim:Zt0-empty}, we have that $Z_t = \emptyset$ for $t \in (t_0, T)$. The intervals $(0,t_0)$ and $(t_0, T)$ may be regarded as empty if $t_0 =0$ or if $t_0 = T$. Thus, we have shown the precise information before and after time $t_0$ described in Theorem \ref{theorem main}. Note that alternatively to~\eqref{t0-embedded-flow}, we could take
$t_0
:= \inf\set{t\in[0,T):
\cH^{n-1}(Z_t)=0
}.
$
This completes the proof of Theorem \ref{theorem main}.
\end{proof}

\subsection{Smoothly immersed flows}
Next, we deal with smoothly immersed flows with self-intersection.
For an immersion $F\colon M\to\bb R^{n+1},$ we let 
\begin{align*}
    S(F):=
    \set{
    F(x) \in \bR^{n+1} :
    x\in M
    \text{ and }
    F(x)=F(y)
    \text{ for some }
    y\in M\setminus\set x
    }
\end{align*}
be its self-intersection set.
For an immersed mean curvature flow $F_t\colon M\to \bb R^{n+1},$ we define $S_t:=S(F_t).$

Recall the following estimate for smoothly immersed hypersurfaces (see ~\cite[Corollary 3.2.4]{gagehamilton86}).
We let $d_{M}$ be the distance function on $M.$

\begin{lem}\label{lem:immersion-dist-comp}
    Let $F\colon M\to\bb R^{n+1}$ be a smooth compact immersion.
    If it satisfies $|A|\le K,$  then
    \begin{align*}
        \abs{F(x)-F(y)}
        \ge \frac 2K \sin\frac{Kd_{M}(x,y)} 2
    \end{align*}
    for any $x,y\in M$ with $d_{M}(x,y)\le \frac \pi K.$
    In particular, given any $x\in M,$ the set $F^{-1}\pr{F(x)}$ is a finite set.
\end{lem}

In the next proposition, we show that a smooth immersion of a closed manifold can be perturbed to an embedding, as long as the self-intersection set is sufficiently small. 
This proposition will be used in the proof of Theorem~\ref{thm:self-int-main} to control the flow at the time when the self-intersection set has small $\cH^{n-1}$-measure. It will play a similar role to Lemma \ref{lem:low-dim-one-side} in Theorem~\ref{theorem main}.

\begin{prop}\label{perturb-immersed-to-embedded}
    Let $M$ be a smooth closed $n$-manifold, and suppose $F: M \to \bR^{n+1}$ is a smooth immersion such that 
    \begin{align}\label{dimS-n-1}
        \cH^{n-1}\pr{S(F)} = 0.
    \end{align}
    Then, for each $\eps>0$ and $k \in \bN$, there is a smooth embedding $F_{\eps}: M\to \bR^{n+1}$ such that $F_{\eps}$ is $\eps$-close to $F$ in $C^{k}$. 
\end{prop}

Proposition~\ref{perturb-immersed-to-embedded} is sharp in the sense that there are immersions of closed manifolds with 
$\cH^{n-1}\pr{S(F)}>0$
yet the immersion can never be perturbed to be embedded for topological reasons. 
One example is Boy's surface, since $\bR \bP^2$ cannot be embedded in $\bR^3$. 
Moreover, we may loosen the regularity assumptions of the proposition to $F$ being $C^k$, for some $k\geq 2$, and the resulting perturbed embeddings $F^{\ve}$ would be $C^{k-1}$ and $\eps$-close in $C^{k-1}$.

The main idea of the proof of Proposition~\ref{perturb-immersed-to-embedded} is straightforward.
The self-intersecting part can be covered by finitely many open balls.
By condition \eqref{dimS-n-1}, in each ball, the hypersurface can be written as the union of finitely many ``ordered'' graphs. 
Thus, we can perturb the hypersurface so that it becomes an embedding in each of them.
The only subtle part is to make sure that we can do this globally and simultaneously in each ball.
To achieve this, we do it by induction.
We show that the graphicality and the ordering of the graphs in one open set is not influenced when a perturbation in another is small enough.

\begin{proof}
	[Proof of Proposition \ref{perturb-immersed-to-embedded}]
When $n=1$, condition \eqref{dimS-n-1} says $\cH^0(S(F))=0$. This means that there are no self-intersections and hence $F$ is an embedding. 
The result follows trivially in this case. 
We may now assume $n \geq 2$ for the rest of the argument. 

We suppose that $S(F) \neq \emptyset$ since otherwise, the result follows trivially. 
Let $p\in S(F),$ and suppose $F^{-1}(p)=\set{x_1,\cdots,x_k}$ with $x_i\neq x_j$ for $i\neq j.$ 
	We note that $k = k(p)$ may depend on $p$.
	Since an immersion is a local embedding, we can find $s_i>0$ such that $F\pr{B_{s_i}(x_i)}$ is embedded for each $i=1,\cdots,k.$ By condition~\eqref{dimS-n-1}, the tangent plane of $F(M)$ coincide for all $x_i$ since they cannot intersect transversely. 
	After shrinking $s_i$'s and $r,$ we may assume 
	\begin{align}\label{CPr-preimage}
		F(M)\cap C^P_r(p)
		= \bigcup_{i=1}^k F\pr{B_{s_i}(x_i)},
	\end{align}
	where $P=P(p)=T_{x_1}M$ is the tangent plane of $F(M)$ at $x_1$ and hence at any~$x_i$, and where
	\begin{align*}
		C^P_r(p)
		:= \set{x\in \bb R^{n+1}: x=q+s\N_P, \text{ where }q \in P \cap B_r(p), |s| < cr
		}
	\end{align*}
	is a cylinder based on $P\cap B_r(p)$ with $c<\infty$ a constant depending on the global curvature bound on $F$. 
	We note that $r = r(p)$ depends on $p$ as well.
	For each $i,$ we can then write $F\pr{B_{s_i}(x_i)}$ as the graph of a function $u_i\colon P\cap B_r(p)\to\bb R$
	with $|\n u_i|<\infty$, possibly shrinking $s_i$ and $r$.
	The finiteness of $|\n u_i|$ implies that the graph map 
	\begin{align*}
		q\in P\cap B_r(p)
		\mapsto
		q + u(q) \N_P
		\in {\rm graph}~u_i
	\end{align*}
	is bi-Lipschitz,
    so the measure bound~\eqref{dimS-n-1} is equivalent to
	\begin{align}\label{equation nodal set dimension bound}
		\cH^{n-1}\pr{\set{q\in P\cap B_r(p): u_i(q) = u_j(q)}}
		=0
	\end{align}
	for all $i\neq j.$
	Lemma~\ref{lemma HJ one side} then implies that $u_i\le u_j$ or $u_j\le u_i$ for any given $i$ and $j.$ 
	After relabelling, we may assume 
	\begin{align}\label{u_k-ordered}
		u_1\le \cdots\le u_k \text{ on }P\cap B_r(p).
	\end{align}
	By~\eqref{equation nodal set dimension bound} and Claim \ref{claim:path-conn} (which is a nontrivial statement when $n \geq 2$), there is a connected subset $\ovl P=\ovl P(p)\sbst P \cap B_r(p)$ where all the inequalities in~\eqref{u_k-ordered} are strict. 
    That is,
    \begin{align}\label{u_k-ordered-strict}
		u_1< \cdots< u_k \text{ on }\ovl P.
	\end{align}	
    In fact, by ~\eqref{equation nodal set dimension bound}, we may choose $\ovl P$ to be dense in $P \cap B_r(p)$. 
    
    Since the graph of $u_i$ corresponds to $F(B_{s_i}(x_i))$, the ordering of the graphs in ~\eqref{u_k-ordered} corresponds to an ordering of $F^{-1}(p)$:
    \begin{align} \label{eqn:orderingofx_k}
    (x_1, \dots, x_k).
    \end{align} 
Due to the denseness and connectedness of $\ovl{P}$, there are only two possible orderings of $F^{-1}(p)$ corresponding to~\eqref{u_k-ordered} and ~\eqref{u_k-ordered-strict}: $(x_1, \dots, x_k)$ and $(x_k, \dots, x_1)$. The former is the ordering of~\eqref{eqn:orderingofx_k} whereas the latter is the ordering where the ``bottom'' sheet is considered the top and vice versa. 

We will talk about preservation of this ordering when we perturb the map $F,$ so we make it rigorous here.
Consider a smooth immersion $G: M \to \bR^{n+1}$ such that for each $x_i  \in F^{-1}(p)$, $G(B_{s_i}(x_i))$ is a local embedding given by the graph $g_i: P\cap B_r(p)\to\bb R$ for each $i$. If the graphs $g_i$ satisfy $g_1 \leq \cdots \leq g_k$ whose ordering is corresponds to the same ordered $k$-tuple $(x_1, \dots, x_k)$ as for $F^{-1}(p)$ in~\eqref{eqn:orderingofx_k}, then we say that $G$ (or its image) has the \textit{original ordering} of $F^{-1}(p)$.
	
	We now consider the open cover of $S(F)$ given by $\cU = \set{C_{r/2}^{P}(p): p \in S(F)}$ where we abbreviate $C_{r(p)/2}^{P(p)}(p) =: C_{r/2}^{P}(p).$ 
	Since $M$ is compact, we have that $F$ is a closed map. 
	In particular, this implies that $S(F)\subseteq \bR^{n+1}$ is a compact subset. 
	Thus, there is a finite set $\cP := \{p_1, \dots, p_N\}\subseteq S(F)$ where $N=|\cP|$ such that $\wt{\cU} = \set{C_{r/2}^{P}(p) \,:\, p \in \cP}$ is a finite subcover of $\cU$ that still covers $S(F).$

For each $p_j \in \cP$, there is an ordering of its local graphs as in~\eqref{u_k-ordered}, which corresponds to an ordering of $F^{-1}(p_j)$, as in~\eqref{eqn:orderingofx_k}. We will perturb $F$ in a neighborhood of $\cP$, and this will come from a graphical perturbation over each $P(p_j)$. We will construct the perturbation of $F$ in a way that preserves the original ordering of $F^{-1}(p_j)$ for each $j$.

	For $\ve>0$ and $p = p_j \in \cP,$ we will construct a perturbation in $C_{r}^P(p)$ (not only $C_{r/2}^{P}(p)$) depending on $\ve$, where $r=r(p_j)$ and $P=P(p_j)$ are taken so that \eqref{CPr-preimage} is true for $x_i\in F^{-1}(p),$ $i=1,\cdots,k(p).$
	Take a non-negative cutoff function $\varphi \colon C_{r}^P(p)\to\bb R$ such that $\varphi|_{C^P_{r/2}(p)}=1$ and $\varphi|_{C^P_r(p)\setminus C^P_{3r/4}(p)}=0.$
	Then, consider the following vector field of $F$,
	\begin{align}\label{V-var-def}
		V_{\ve}^p(x)
		:= \begin{cases}
			0&\text{if }F(x)\in M\setminus C^P_r(p)\\
			(i-1)\cdot \varepsilon \cdot \varphi(F(x)) \N(x_i)&
			\text{if }x\in B_{s_i}(x_i)
			\text{ for some }i\in\set{1,\cdots,k(p)},
		\end{cases}
	\end{align}
	where $\N(\cdot)$ is a choice (also depending on $p$) of unit normal vector field of $F$ in a neighborhood of $B_{s_i}(x_i)$ for each $i$, such that $\N(x_i)$ all point in the same direction in $\bR^{n+1}$. The perturbation defined in~\eqref{V-var-def} is not a normal perturbation but rather a perturbation in the direction of the normal $\N(x_i)$ to the plane $P$, which makes the construction somewhat simpler. 
    
    We specify the choice of $\N$ as follows.
	Recall that $\N\pr{x_i}$'s are all parallel since condition ~\eqref{dimS-n-1} implies that they cannot intersect transversely.
    By \eqref{u_k-ordered-strict}, we may choose $q \in \ovl P(p)$ such that $u_2(q) > u_1(q)$. 
	Then, let $x_q \in \mathrm{graph}(u_1)$ and $y_q \in \mathrm{graph}(u_2)$ be points in $\bR^{n+1}$ corresponding to $u_1(q)$ and $u_2(q)$, respectively. 
	We choose $\N$ such that $\langle \N, y_q - x_q\rangle >0$. 
	In other words, we choose $\N$ such that $\N(x_i)$ is oriented into the side of $\mathrm{graph}(u_1)$ that $\mathrm{graph}(u_2)$ lies in.

	Recall that we have fixed the orders of $x_i$'s and $u_i$'s such that \eqref{u_k-ordered} is true.
    A subtle point is that this ordering may ``switch'' among different cylinders in the open cover $\wt{\cU}$.

	We consider
	$V_\varepsilon
	:= \sum_{j=1}^N V_{\ve}^{p_j},$
	which satisfies $|V_\varepsilon|\le c_0\varepsilon$ for some $c_0=c_0(F)<\infty.$
	We claim that for all small enough $\ve>0$,
	\begin{align}
		\label{Fs-emb}
		F_\varepsilon
		:= F + V_\varepsilon\,
		\text{ is an embedded hypersurface.}
	\end{align}
	For each $j,$ we let $r_j:=r(p_j)$ and $P_j:=P(p_j).$
	We will work with the cover $\set{C^{P_j}_{r_j}(p_j): 1 \leq j \leq N}$, which is the same as $\wt{\cU}$ but has cylinders of twice the scales used in $\wt{\cU}$.

	We will show \eqref{Fs-emb} by working on each $V_{\varepsilon}^{p_j}$ inductively.
	We let $C_j := C^{P_j}_{r_j}(p_j)$ and $C_j' := C^{P_j}_{r_j/2}(p_j).$
	First, by the definition of $V_{\varepsilon}^{p_1},$ it is clear that 
    \begin{align}\label{F-emb-ind1-1}
		F+V_{\varepsilon_1}^{p_1}
		\text{ is an embedding on }F^{-1}\pr{C'_1}
	\end{align}
	for small $\varepsilon_1>0.$
	Also, when $\varepsilon_1$ is small, 
	\begin{align}\label{F-emb-ind2-1}
		\text{for each $j$, } \pr{F+V_{\varepsilon_1}^{p_1}}\pr{F^{-1}\pr{C_j}}
		\text{ is graphical over }P_j
		\text{ with the original ordering of } F^{-1}(p_j).
	\end{align}
  By ``graphical over $P_j$'' in~\eqref{F-emb-ind2-1} and its subsequent analogues, we mean that $F+V_{\eps_1}^{p_1}$ restricted to each component of $F^{-1}(C_j)$ is graphical over $P_j$.
	To justify~\eqref{F-emb-ind2-1} for $C_j,$ we first see that the definition of $V^{p_1}_\varepsilon$ implies
    \begin{align}\label{F-same-order-C1}
    \pr{F+V_{\varepsilon_1}^{p_1}}\pr{F^{-1}\pr{C_1}}
		\text{ is graphical over }P_1
		\text{ with the original ordering of } F^{-1}(p_1),
    \end{align}
    when $\varepsilon_1$ is small, since $F+V_{\varepsilon_1}^{p_1}$ is an embedding on $F^{-1}\pr{C^{P_1}_{3r_1/4}(p_1)}$ and since we have that $F+V_{\varepsilon_1}^{p_1} = F$ on $F^{-1}\pr{\bR^{n+1}\setminus C^{P_1}_{3r_1/4}(p_1)}$. 
    This proves \eqref{F-emb-ind2-1} when $j=1.$
    For $j\ge 2,$
    we use the ordering of \eqref{u_k-ordered} and write
	\begin{align*}
		F(M)\cap C_j
		= \bigcup_{i=1}^{k_j}
		{\rm graph}_{P_j} u_{ji},
	\end{align*}
	where $u_{ji}\colon P_j\cap B_{r_j}(p_j)\to\bb R$ satisfies
    $u_{j1}\le\cdots\le u_{jk_j}.$ 
    We consider $j=2$ since the proof of ~\eqref{F-emb-ind2-1} works the same for arbitrary $j\geq 2$. Suppose $F(M)\cap C_1\cap C_2\neq \emptyset.$ We let $\N_1$ and $\N_2$ be the local unit normals of $F$ in $C_1$ and $C_2$ chosen above.
	Then for $i\in\set{1,\cdots,k_1}$ such that ${\rm graph}_{P_1}u_{1i}\cap C_2\neq \emptyset,$ we have
	\begin{align}\label{u1i-graph}
		\ppair{\N_1, \N_2(p_2)}
        \neq 0
        \text{ on }F^{-1}\pr{{\rm graph}_{P_1}u_{1i}\cap C_2},
	\end{align}
    after possibly refining the finite open cover $\cP$ and shrinking all $r(p_j)$. 
	Thus, when $\varepsilon_1$ is small, depending on the infimum of $\abs{\ppair{\N_1, \N_2(p_2)}}$ on $F^{-1}\pr{{\rm graph}_{P_1}u_{1i}\cap C_2},$ the condition \eqref{u1i-graph} is preserved after being perturbed by $V^{p_1}_{\varepsilon_1},$ so ${\rm graph}_{P_1}u_{1i}\cap C_2$ is still graphical over $P_2$ after the perturbation. 
    This implies that $\pr{F+V_{\varepsilon_1}^{p_1}}\pr{F^{-1}\pr{C_2}}$ is graphical over $P_2,$ completing the first part of \eqref{F-emb-ind2-1} when $j=2.$

    To check the ordering of the perturbed graphs, because $F+V^{p_1}_{\varepsilon_1} = F$ on $F^{-1}\pr{\bR^{n+1}\setminus C^{P_1}_{3r_1/4}(p_1)},$ it suffices to check the part $F^{-1}\pr{C_2\cap C^{P_1}_{3r_1/4}(p_1)}.$
    Using the facts
    \begin{enumerate}
        \item\label{preserve-order-1}
        that $F+V^{p_1}_{\varepsilon_1}$ is an embedding on $F^{-1}\pr{C_2\cap C^{P_1}_{3r_1/4}(p_1)}$ for any small enough $\varepsilon_1,$ and
        \item\label{preserve-order-2}
        that $F$ is an embedding away from an $\cH^{n-1}$-measure zero subset in $F^{-1}\pr{C_2\cap C^{P_1}_{3r_1/4}(p_1)},$
    \end{enumerate}
    we can show that the original ordering of $F^{-1}(p_2)$ is preserved over $P_2$ after the perturbation of $V_{\varepsilon_1}^{p_1}.$
    In fact, given any $q\in P_2$ such that $q+u_{2i}(q)\N_2\in F^{-1}\pr{C_2\cap C^{P_1}_{3r_1/4}(p_1)}$ for some $i\in \set{1,\cdots, k_2},$ by \eqref{preserve-order-2}, we can find a point $q'\in P_2$ that is arbitrarily close to $q$ such that 
    \begin{align}\label{preserve-order-uq}
    u_{2(i-1)}(q')
    < u_{2i}(q')
    < u_{2(i+1)}(q')
    \end{align}
    whenever $i-1,i+1\in \set{1,\cdots, k_2}.$
    By the continuity of the family of the vector fields $V^{p_1}_{\varepsilon_1}$ in $\varepsilon_1$ and \eqref{preserve-order-1}, the strict inequalities \eqref{preserve-order-uq} remain true after the perturbation by $V_{\varepsilon_1}^{p_1}$ for any small enough $\varepsilon_1.$
    This preservation can be done for a point arbitrarily close to $q,$ and hence the continuity implies that the preservation of the original ordering is also true for $q\in P_2.$
    This proves the second part of \eqref{F-emb-ind2-1} when $j=2.$
    The same argument shows that~\eqref{F-emb-ind2-1} holds for any $j \geq 2$.

    We proceed to prove~\eqref{Fs-emb} using an induction argument whose base case is justified by~\eqref{F-emb-ind1-1} and~\eqref{F-emb-ind2-1}.
	Suppose for some $\ell>0,$ we have that 
	\begin{align}\label{F-emb-ind1-2}
		F+\sum_{j'=1}^\ell V_{\varepsilon_{j'}}^{p_{j'}}
		\text{ is an embedding on }
		F^{-1}\pr{
		\bigcup_{j'=1}^\ell C_{j'}'}
	\end{align}
	for some positive small $\varepsilon_i$'s and that
	\begin{align}\label{F-emb-ind2-2}
		\text{for each $j$}, \pr{F+\sum_{j'=1}^\ell V_{\varepsilon_{j'}}^{p_{j'}}}\pr{F^{-1}\pr{C_{j}}}
		\text{ is graphical over }P_j
		\text{ with the original ordering of } F^{-1}(p_j).
	\end{align}
    Then to prove the statement for $\ell+1$ in \eqref{F-emb-ind1-2}, it suffices to prove 
    \begin{align*}
		F+\sum_{j'=1}^{\ell+1} V_{\varepsilon_{j'}}^{p_{j'}}
		\text{ is an embedding on }
		F^{-1}\pr{C_{\ell+1}'}
	\end{align*}
    for small enough $\varepsilon_{\ell+1}>0.$
    This then follows directly from the definition of $V^{p_{\ell+1}}_{\varepsilon_{\ell+1}}$
    because based on \eqref{F-emb-ind2-2} applied to $j=\ell+1,$ $\pr{F+\sum_{j'=1}^\ell V_{\varepsilon_{j'}}^{p_{j'}}}\pr{F^{-1}\pr{C_{\ell+1}}}$ is still a union of graphs with the original ordering of $F^{-1}\pr{p_{\ell+1}}$.
    For \eqref{F-emb-ind2-2} with $\ell$ replaced with $\ell+1,$ the same argument in the preceding paragraph suffices.
    Thus, combining this with \eqref{F-emb-ind1-1} and \eqref{F-emb-ind2-1}, we can use induction to prove 
	\begin{align*}
		F+\sum_{j=1}^N V_{\varepsilon_j}^{p_j}
		\text{ is an embedding on }
		F^{-1}\pr{
			\bigcup_{j=1}^N C_j'}.
	\end{align*}
	Since $C_j'$'s cover $S(F),$ $F+\sum_{j=1}^N V_{\varepsilon_j}^{p_j}$ is an embedding on $F^{-1}\pr{S(F)}.$
	Moreover, when proving \eqref{F-emb-ind2-1} and \eqref{F-emb-ind2-2}, we also obtain that at a point where the inequalities are strict, the strictness is preserved after the perturbation.
	Thus, $F+\sum_{j=1}^N V_{\varepsilon_j}^{p_j}$ is an embedding on $M\setminus F^{-1}\pr{S(F)}.$
	This concludes the proof of the proposition by taking $\varepsilon:= 
    \min \set{\varepsilon_j:j=1,\cdots,N}.$
\end{proof}

Proposition~\ref{perturb-immersed-to-embedded} deals with the critical case of Theorem~\ref{thm:self-int-main}.
We use it to complete the proof of Theorem~\ref{thm:self-int-main}.

\begin{proof}
[Proof of Theorem~\ref{thm:self-int-main}]
    If $S_0=\emptyset,$ then the result follows from the preservation of embeddedness of smooth mean curvature flows.
    Thus, we assume $S_0\neq\emptyset.$

    For $p\in S_0,$ assume $p=F_0(x_1)=F_0(x_2)$ for some $x_1\neq x_2$ in $M.$
    By Lemma~\ref{lem:immersion-dist-comp} and the interior estimate~\cite{EckerHuisken91}, we can find $r>0$ and $\delta>0$ such that $F_t\pr{B_r(x_1)}$ and $F_t\pr{B_r(x_2)}$ are both embedded mean curvature flows for $t\in[0,\delta].$
    Thus, by Theorem~\ref{thm-local-bound}, we obtain
    \begin{align*}
        \cH^{n-1}\pr{
        F_t\pr{B_r(x_1)} \cap F_t\pr{B_r(x_2)}
        }
        <\infty
    \end{align*}
    for $t\in (0,\delta).$
    This argument works for any $x_1,x_2\in F_0^{-1}\pr{p}.$
    Since there are only finitely many points in $F_0^{-1}\pr{p}$ and since $M$ is compact, this implies 
    \begin{align}\label{S_t-measure-bound}
        \cH^{n-1}\pr{S_t}<\infty
    \end{align} 
    for $t>0.$ This proves the first conclusion of Theorem~\ref{thm:self-int-main}.

    As in the proof of Theorem~\ref{theorem main}, if $\dim(S_0)\ge n-1,$ then \eqref{S_t-measure-bound} implies $\dim S_t\le n-1\le \dim S_0.$
    We will now consider the possibility of $\dim S_{t_0} < n-1$ for some $t_0 \geq 0$. The following claim will imply that if $\dim S_{t_0}< n-1$ for some $t_0\geq 0$, then $S_t=\emptyset$ for all $t>t_0$.
    This will be used to conclude the dimension monotonicity result of Theorem \ref{thm:self-int-main}.

\begin{claim}
\label{claim:St0-empty}
    If $\cH^{n-1}\pr{S_{t_0}}=0$ for some $t_0 \in [0,T)$, then $S_t = \emptyset$ for all $t \in (t_0, T)$.
\end{claim}

\begin{proof}[Proof of Claim~\ref{claim:St0-empty}]    

Without loss of generality, we will consider $t_0 = 0$, so we assume $\cH^{n-1}(S_0)=0$. 

    Let $p\in S_0,$ and suppose $F_0^{-1}(p)=\set{x_1,\cdots,x_k}$ with $x_i\neq x_j$ for $i\neq j.$ Suppose $M_t$ is given by the image of a family $F_t\colon M\to\bb R^{n+1}.$ 
    As in the first paragraph of the proof of Proposition~\ref{perturb-immersed-to-embedded}, we can find $r>0,$ a hyperplane $P,$ and functions $u_i\colon P\cap B_r(p)\times[0,\delta] \to\bb R$ such that
    \begin{align*}
        M_t\cap C^P_r(p)
        = \bigcup_{i=1}^k {\rm graph} ~u_i(\cdot, t)
    \end{align*}
    for $t\in [0,\delta]$ for some $\delta>0$ with
    \begin{align*}
        u_1(\cdot,0)\le \cdots\le u_k(\cdot,0) \text{ on }P\cap B_r(p).
    \end{align*}
    By Proposition~\ref{perturb-immersed-to-embedded} and the assumption that $\cH^n(S_0)=0$, we may find a smooth embedding $F_{\ve}$ which is $\varepsilon$-close to $F$ in $C^2$. 
    If $(F_{\varepsilon})_t$ is the smooth mean curvature flow starting from $(F_{\varepsilon})_0 = F_{\ve},$ then $(F_{\ve})_t$ is embedded as long as it exists.
    By the compactness of $M,$ we can find $\delta>0$ such that $(F_\varepsilon)_t$ exists for $t\in[0,\delta]$ for all $\varepsilon>0$ small enough.
    From the construction of $F_\varepsilon,$ for such an $\varepsilon,$ we can write
    \begin{align*}
        \pr{F_\varepsilon}_t(M) \cap C^P_r(p)
        = \bigcup_{i=1}^k {\rm graph} ~u_{i,\varepsilon}(\cdot, t)
    \end{align*}
    where the functions $u_{i,\varepsilon}\colon \pr{P\cap B_r(p)} \times[0,\delta] \to\bb R$ converge to $u_i$ as $\varepsilon\to 0$ in $C^2_{\rm loc}\pr{\pr{P\cap B_r(p)}\times[0,\delta]}$ and satisfy
    \begin{align*}
        u_{1,\varepsilon}< \cdots< u_{k,\varepsilon} \text{ on }P\cap B_r(p)\times[0,\delta]
    \end{align*}
    by the embeddedness of $\pr{F_\varepsilon}_t.$
    Taking $\varepsilon\to 0,$ we get 
    \begin{align*}
        u_{1}\le \cdots\le u_{k} \text{ on }P\cap B_r(p)\times[0,\delta].
    \end{align*}
    Thus, as in the proof of Theorem~\ref{theorem main}, we can apply the strong maximum principle on each $P\cap B_r(p)$ to see that 
    \begin{align*}
        u_1< \cdots < u_k \text{ on }\pr{P\cap B_r(p)}\times (0,\delta).
    \end{align*}
    This means that for $t>0,$ $F_t$ is an embedding and hence $S_t=\emptyset.$ This completes the proof of the claim.
\end{proof}
    
    As in Theorem \ref{theorem main}, combining~\eqref{S_t-measure-bound} with Claim~\ref{claim:St0-empty} completes the proof of the dimension monotonicity result and implies the precise information before and after time $t_0:= \inf \set{t\in[0,T): \dim S_t < n-1} \in [0,T]$ described in Theorem~\ref{thm:self-int-main}.
    This completes the proof of Theorem~\ref{thm:self-int-main}.
\end{proof}


\section{Dimension Monotonicity of the Intersection of Weak Solutions}\label{section:dimension-monotonicity-weak}

We will prove Theorem~\ref{thm:dim-mono-nonfattening} and its generalizations in this section.
To this end, we work with Brakke flows and level set flows.
For introductions to Brakke flows, level set flows, and other weak formulations of mean curvature flow, we refer to \cite{Ilmanen94}.

\textbf{Assumption:} 
Throughout this paper, we will make the implicit assumption that every Brakke flow $\{\mu_t\}_{t \geq 0}$ has bounded area ratios.
That is,
\begin{align*}\sup_{t \geq 0} \sup_{r >0} \sup_{x \in \bR^{n+1}} \frac{\mu_t(B_r(x))}{r^n}< \infty.\end{align*}
The assumption of having bounded area ratios allows us to take tangent flows at any point. 
This is a mild assumption, which holds whenever the Brakke flow has a smooth, closed embedded initial condition. 

\subsection{Preliminaries on Brakke flows and level set flow}

In our first theorem, we notice that the local finiteness of $\cH^{n-1}$-measure from Theorem~\ref{thm-local-bound} applies immediately to the regular parts of intersecting Brakke flows.
By the regular part of a Brakke flow $\reg \cM_t$, we mean the set of points in the spacetime support of the flow $(\spt \cM)_t$ such that in a small forwards and backwards parabolic neighborhood of the point, $\cM_t = \cH^{n-1}\lfloor M_t$ for some smooth flow $M_t$. 
The regular part of a unit regular Brakke flow is the same as the set of points admitting a multiplicity one planar tangent flow.
Theorem~\ref{theorem local measure estimate} follows from Theorem~\ref{thm-local-bound} and the fact that the regular part of a unit density Brakke flow is spatially real analytic, so local coincidence of regular parts implies coincidence of entire connected components of regular parts.

\begin{thm}\label{theorem local measure estimate}
	If $\cM_t$ and $\cN_t$ are unit regular, integral n-dimensional Brakke flows in $\bb R^{n+1}$ for $t\in [0,\infty),$ then for each $t\in (0,\infty),$ one of the following holds:
 \begin{enumerate}
      \item The Hausdorff dimension of $\reg \cM_t\cap \reg \cN_t$ is at most $n-1$.
     \item A connected component of $\reg \cM_t$ is a subset of $\reg \cN_t$, or vice versa. 
 \end{enumerate}
\end{thm}

We now recall the definitions of non-fattening and non-discrepancy of level set flows.
Given a closed set $M \subseteq \bR^{n+1},$ we define $F_t(M)$ to be the level set flow of $M$ at time $t$, with initial condition $F_0(M) = M$.

\begin{defn}\label{def:non-dis}
Let $M \subseteq \bR^{n+1}$ be a closed set such that $\cH^{n}(M) < \infty$. 
We say that the level set flow of $M$ is \textbf{non-fattening} if $\cH^{n+1}\pr{F_t(M)}=0$ for all $t \geq 0$.

Suppose in addition that there is a compact domain $K\subseteq \bR^{n+1}$ such that $\dd K = M$.
We define
\begin{align*}
    \cK &:= \{(x,t) \in \bR^{n+1} \times [0, \infty): x \in F_t(K)\},\text{and}\\
    \cK' &:= \set{(x,t) \in \bR^{n+1} \times [0, \infty): x \in F_t\pr{\ovl{\bR^{n+1}\setminus K}}}.
\end{align*}
The outer flow and the inner flow of $M$ are defined by 
\begin{align*}
    t &\mapsto M^{\rm out}_t:= \{x \in \bR^{n+1}: (x,t) \in \dd \cK\},\text{and}\\
    t &\mapsto M^{\rm in}_t:= \{x \in \bR^{n+1} : (x,t)\in \dd \cK'\}.
\end{align*}
Here, $\dd \cK$ and $\dd \cK'$ denote the relative boundaries of $\cK$ and $\cK'$ in spacetime $\bR^{n+1} \times \bR_{\geq 0}$.
We say that the level set flow of $M$ is \textbf{non-discrepant} if $F_t(M) = M^{\rm out}_t = M^{\rm in}_t$ for all $t \geq 0$. 
\end{defn}

Before they were defined in the current form by Hershkovits--White in~\cite{HershkovitsWhite17}, the notions of inner and outer flows had been implicitly used, for example, in~\cite{AAG95, Ilmanen94}.
A key observation by Hershkovits--White is that non-discrepant level set flows are in particular non-fattening (see~\cite[Remark A.3]{HershkovitsWhite17}).

For a compact, non-fattening level set flow $M_t,$ Ilmanen~\cite{Ilmanen94} proved that there is a unique unit-density Brakke flow $\cM_t$ such that $\spt \cM_0 = M_0.$
This is called a boundary motion in ~\cite{Ilmanen94}. 
We note that the boundary motions we use will implicitly be compact, unless otherwise stated.

We now state a couple general results for level set flows.
The first one is the following fact about weak set flows and level set flows proven by Ilmanen. Recall that a weak set flow is a set-theoretic subsolution of mean curvature flow, meaning that it satisfies the avoidance principle with respect to all closed MCFs (see~\cite{Ilmanen93, White95}).

\begin{lem}
    [{\cite[Principle~4.D]{Ilmanen93}}]
    \label{lem:Ilm-contain}
    Let $P_t$ be a weak set flow and $Q_t$ be a level set flow in $\bb R^{n+1}$ for $t\in [0,\infty).$
    If $P_0\sbst Q_0,$ then $P_t\sbst Q_t$ for all $t\ge 0.$
\end{lem}

In Definition~\ref{def:lsf-sing}, we will define what we mean by a singularity for a non-fattening level set flow.
To this end, we show that a non-fattening level set flow is the union of its inner and outer flows.

\begin{lem}
\label{lem:nonfat-in+out}
    Let $M$ be a closed smooth hypersurface in $\bb R^{n+1}$ and $M_t$ be the level set flow starting from $M.$
    If $M_t$ is non-fattening, then $M_t = M^{\rm out}_t\cup M^{\rm in}_t.$
\end{lem}

\begin{proof}
    We let $M=\dd D$ where $D$ is the compact domain bounded by $M,$ and let $u\colon \bb R^{n+1}\times \bb R\to\bb R$ be a level set function such that $M_t = \set{u(\cdot,t)=0}$ and $\dd\set{u\ge 0}$ and $\dd\set{u\le 0}$ are the spacetime tracks of $M^{\rm out}_t$ and $M^{\rm in}_t.$ We note that the level set function $u$ is continuous in spacetime. 
    
    Now, we may decompose $\bR^{n+1}\times \bR$ into a union of the following disjoint sets: $\inte \{u \geq 0\},$  $\dd \{u \geq 0\}$, and $\inte \Big((\bR^{n+1}\times \bR)\setminus \{u \geq 0\}\Big) = \{u <0\}$. Likewise, $\bR^{n+1}\times \bR$ may be decomposed into a union of the disjoint sets $\inte \{u \leq 0\},$  $\dd \{u \leq 0\}$, and $\{u >0\}$. 
    
    Suppose for a contradiction that $M^{\rm out}_{t_0}\cup M^{\rm in}_{t_0}\subsetneq M_{t_0}$ for some $t_0\ge 0.$
    Then we can take $x_0\in M_{t_0}\setminus\pr{M^{\rm out}_{t_0}\cup M^{\rm in}_{t_0}}.$
    This means that $(x_0, t_0)$ does not belong to $\dd\{u \geq 0\}$ or $\dd\{u \leq 0\}$. Thus,
    \begin{align*}
        (x_0, t_0)
        \in \pr{\inte\set{u\ge 0}\cup \set{u< 0}}
        \cap \pr{\inte\set{u\le 0}\cup \set{u> 0}}.
    \end{align*}
    Since $x_0 \in M_{t_0},$ we have that $u(x_0, t_0)=0$. This means $(x_0, t_0)$ is an element of neither $\set{u < 0}$ nor $\set{u > 0}.$ Hence, 
    \begin{align*}
    (x_0, t_0)\in \inte\set{u\ge 0} \cap \inte\set{u\le 0}.
    \end{align*}
    Since this is a nontrivial intersection of open sets, we can find $r>0$ and an open ball $B_r(x_0)$ in $\bb R^{n+1}$ such that
    \begin{align*}
        B_r(x_0)
        \sbst \set{u\ge 0}\cap \set{u\le 0}
        \cap \set{t=t_0}
        = \set{u(\cdot,t_0) = 0}.
    \end{align*}
    This implies that $M_t$ is fattening at time $t_0$, a contradiction.
    This shows that $M_{t_0}\setminus\pr{M^{\rm out}_{t_0}\cup M^{\rm in}_{t_0}} = \emptyset$. Since the inner and outer flows are always subsets of the level set flow $M_t$ (see~\cite[Corollary A.4]{HershkovitsWhite17}), we conclude the lemma. 
\end{proof}

Based on Lemma~\ref{lem:nonfat-in+out}, we define singularities in the following way.
The definition of regular points is the same as the one used in \cite{HershkovitsWhite17}.

\begin{defn}
\label{def:lsf-sing}
    Let $X_t$ be a weak set flow.
    We say that $(x_0, t_0)$ is a regular point of $X_t$ if $x_0\in X_{t_0}$ and there exists $\delta>0$ such that the flow
    \begin{align*}
        t\mapsto X_t\cap B_\delta(x_0)
    \end{align*}
    is a smooth mean curvature flow of smooth, properly embedded hypersurfaces for $t\in \pr{t_0-\delta^2, t_0+\delta^2}.$
    If $x_0\in X_{t_0}$ and $(x_0,t_0)$ is not a regular point of $X_t,$ we say that $(x_0, t_0)$ is a singular point of $X_t.$

    Given a level set flow $M_t,$ a point $(x_0, t_0)$ is called a singularity of $M_t$ if either $(x_0, t_0)$ is a singular point of $M^{\rm out}_t$ or $(x_0, t_0)$ is a singular point of $M^{\rm in}_t.$
\end{defn}

We give two remarks here.
First, Definition \ref{def:lsf-sing} may at first glance look too strong to allow some potential non-fattening discrepant flows, since regular points require the flow to be smooth in a full forwards and backwards spacetime neighborhood.
However, this is a natural condition that arises when there are only finitely many singular times for a non-fattening flow. 
When an LSF is non-fattening and has only finitely many singular times, it turns out that backwards regularity of a spacetime neighborhood of a point in the inner/outer flow implies regularity of the inner/outer flow in a full spacetime neighborhood. 
This follows from~\cite[Theorem~B.6]{HershkovitsWhite17} after checking that the finiteness of singular times implies that the inner/outer flows lie in the closure of the interior of an appropriate set. 
Since this result is not needed in this paper, we do not include the details. Second, we note that we do not use the entirety of the level set flow to define singularities but just the inner and outer flows.
This allows us to include some interesting examples when we make an assumption of finiteness of singularities, especially examples of fattening level set flows.
For instance, consider a shrinker asymptotic to a regular cone which has a fattening level set flow. 
If the inner and outer flows of the cone are smooth for a short time after the first singular time, then its level set flow would have only one singularity, based on Definition~\ref{def:lsf-sing}, up through slightly after the first singular time. 
The assumption on the inner and outer flows is natural in light of work of Chodosh--Daniels-Holgate--Schulze~\cite[Theorem~1.2]{CDS23}.


\subsection{Localizable weak solutions}\label{section splittable and localizable}

In this section, we will prove dimension monotonicity results for intersection of certain Brakke flows and level set flows.
In the following definition, we note that we may sum Brakke flows by summing the corresponding Radon measures.

\begin{defn}\label{definition-splittable}
    We say that an integral $n$-dimensional Brakke flow $\{\cM_t\}_{t \geq 0}$ in $\bR^{n+1}$ is \textbf{localizable} if it satisfies the following property: If $K\subseteq \bR^{n+1}$ is a smooth closed domain such that $\dim(\dd K \cap (\spt \cM)_{t_0}) < n-1$ for some $t_0\geq 0$, then there exist integral $n$-dimensional Brakke flows $\{\cM^1_t\}_{t\geq t_0}$ and $\{\cM^2_t\}_{t \geq t_0}$ such that 
    \begin{enumerate}
        \item \label{loc-BF-1}
        $\pr{\spt \cM^1}_{t_0} = (\spt \cM)_{t_0} \cap K$ and $\pr{\spt \cM^2}_{t_0} = (\spt \cM)_{t_0} \cap \ov{\bR^{n+1}\setminus K}$,
        \item \label{loc-BF-2}
        $\cM_t = \cM^1_t +\cM^2_t$ for $t\geq t_0.$
    \end{enumerate}
\end{defn}

Given a Brakke flow $\cM_t,$ we define $(\spt \cM)_t$ to be the time $t$-slice of the spacetime support of $\cM_t$. This is possibly distinct from $\spt \cM_t$. We also define $\sing \cM_t$ to be all the points of $\bR^{n+1}$ such that $\cM_t$ has a nontrivial tangent flow at time $t$ which is not a static multiplicity one plane. 
We note that it is possible for $x \in \sing \cM_t$ yet $x \notin \spt \cM_t$.

In the following theorem, we show that a localizable Brakke flow, subject to a multiplicity assumption, has monotone intersection dimension with smooth flows unless their regular parts coincide.
We note that the multiplicity assumption is a natural condition which is conjectured to hold true for Brakke flows starting at smooth closed hypersurfaces. 
This is known to be true for mean convex flows in $\bR^{n+1}$ by~\cite{Wh03} and for arbitrary flows in $\bR^3$ by~\cite{BK23}.

\begin{thm}\label{theorem-splittable-implies-GAP}
    Let $\cM_t$ be an integral, unit regular $n$-dimensional Brakke flows in $\bR^{n+1}$ which has no higher multiplicity planar tangent flows. 
    Suppose that $\cM_t$ is localizable. 
    Then, for each smooth closed connected mean curvature flow of hypersurfaces $N_t$, one of the following conditions holds:
    \begin{enumerate}
        \item \label{item:splittableGAP1} For some $t \in [0, \infty)$, either $N_t$ is a subset of $\reg \cM_t$ or a connected component of $\reg \cM_{t}$ is a subset of $N_t$. 
        \item \label{item:splittableGAP2} The Hausdorff dimension of $(\spt\cM)_t \cap N_t$ is non-increasing for as long as $N_t$ exists. Moreover, if if $N_t$ is any smooth closed connected MCF and if $\dim((\spt \cM))_{t_0}\cap N_{t_0})< n-1$, then $(\spt \cM)_{s} \cap N_s = \emptyset$ for all $s>t_0$.
    \end{enumerate}
\end{thm}

\begin{proof}
    Suppose $N_t$ exists for $t \in [0, T)$. 
    We will treat $N_t$ as a Brakke flow which is smooth for $t \in [0, T)$ and which vanishes for $t>T$, i.e. $N_t = \emptyset$ for $t \in (T, \infty)$. 
    We could also modify the initial condition for $N_t$ to be at any positive time, but we work with initial condition $N_0$ at $t=0$ for simplicity.

To prove that either~\eqref{item:splittableGAP1} or ~\eqref{item:splittableGAP2} holds, we may prove that~\eqref{item:splittableGAP2} holds assuming that for each $t \in [0, \infty)$, $N_t\not\subseteq \reg \cM_t$ and no connected component of $\reg \cM_t$ is a subset of $N_t$. 
    Under this assumption, Theorem \ref{theorem local measure estimate} implies that for $t \in (0,\infty)$,
    \begin{align}\label{equation reg intersection 2}
        \dim(\reg \cM_t \cap N_t) \leq n-1.
    \end{align}
    Since $\cM_t$ has no higher multiplicity planar tangent flows, the Brakke regularity theorem implies that $\sing \cM_t$ corresponds exactly to the points of $\bR^{n+1}$ where the flow has a nontrivial nonplanar tangent flow. 
    By White's stratification theorem for Brakke flows~\cite[Theorem~9]{Wh97}, it follows that 
    \begin{align}\label{equation sing bound 2}
    \dim(\sing \cM_t) \leq n-1\end{align}
    for each $t\in (0, \infty)$. 
    By our definition of $\sing \cM_t$, we have that $(\spt \cM)_t = \reg \cM_t \cup \sing \cM_t$.
    Thus, from~\eqref{equation reg intersection 2} and~\eqref{equation sing bound 2}, we have that for all $t \in (0, \infty)$,
    \begin{align}\label{equation spt bound 2}
        \dim((\spt \cM)_t \cap N_t) \leq n-1.
    \end{align}

Next, we define $t_0 := \inf\{t \in [0, \infty) : \dim((\spt \cM)_t \cap N_t) < n-1\}$. 
We will now prove that for all $t> t_0$,
    \begin{align}\label{goal-Brakke 2}
        (\spt \cM)_t \cap N_t = \emptyset.
    \end{align}
By definition of $t_0$, there exist $t_i \ssearrow t_0$ such that $\dim((\spt \cM)_{t_i}\cap N_{t_i}) < n-1$. 
Let $K_t$ be the smooth closed domain such that $\dd K_t = N_{t}$. 
In particular, we have that $\dim((\spt \cM)_{t_i} \cap \dd K_{t_i})< n-1$ for each $t_i$. 
By the assumption that $\cM_{t}$ is localizable, there exist integral $n$-dimensional Brakke flows $\cM^{1,i}_t$ and $\cM^{2,i}_t$ such that $\cM_t = \cM^{1,i}_t + \cM^{2,i}_t$ for all $t \geq t_i$ and such that $(\spt \cM^{1,i})_{t_i} \subseteq K_{t_i}$, and $(\spt \cM^{2,i})_{t_i} \subseteq \ov{\bR^{n+1}\setminus K_{t_i}}$. 
Since a codimension one Brakke flow is a weak set flow~\cite[10.6]{Ilmanen94}, Lemma \ref{lem:Ilm-contain} implies that $(\spt \cM^{1,i})_t \subseteq K_t$ and $(\spt \cM^{2,i})_t \subseteq \ov{\bR^{n+1}\setminus K_{t}}$ for all $t\geq t_i$.

Since $\cM_t = \cM^{1,i}_t + \cM^{2,i}_t$, we have that $(\spt \cM)_t = (\spt \cM^{1,i})_t \cup (\spt \cM^{2,i})_t$ for each $t \geq t_i$. To prove~\eqref{goal-Brakke 2}, it suffices to show that for each $i$ and all $t>t_i$,
\begin{align}\label{goal Brakke 3}
\begin{aligned}
    (\spt \cM^{1,i})_{t} \cap N_{t} = (\spt \cM^{2,i})_{t} \cap N_{t} &= \emptyset.
\end{aligned}
\end{align}
Indeed, since $(\spt \cM)_t = (\spt \cM^{1,i})_t \cup (\spt \cM^{2,i})_t$, applying~\eqref{goal Brakke 3} to each $i$ would  imply~\eqref{goal-Brakke 2}.

We now prove~\eqref{goal Brakke 3} for each $i$ by contradiction. 
Suppose that $(\spt \cM^{1,i})_{t^*} \cap N_{t^*} \neq \emptyset$, for some $t^*>t_i$, and let $x \in (\spt \cM^{1,i})_{t^*} \cap N_{t^*}$. 
We now consider a tangent flow of $\cM^{1,i}_{t^*}$ at the spacetime point $(x,t^*)$. 
Since $(\spt \cM^{1,i})_t \subseteq K_t$ and since $K_t$ has smooth boundary $N_t$, we find that any tangent flow of $\cM^{1,i}_t$ at $(x,t^*)$ is contained in a halfspace $\bH$. 
Moreover, since $x \in (\spt \cM^{1,i})_{t^*} \cap N_{t^*}$, we have that any tangent flow of $\cM^{1,i}_t$ at $(x,t^*)$ nontrivially intersects $\dd \bH$. 
Now, any tangent flow of $\cM^{1,i}_t$ is $F$-stationary, i.e.,\ is a varifold shrinker, so we can apply Solomon--White's strong maximum principle~\cite{SW89}. 
We find that any tangent flow of $\cM^{1,i}_t$ at $(x,t^*)$ is supported on the hyperplane $\dd \bH$ and has some integer multiplicity.

\textbf{Case 1}: $x \notin  (\spt \cM^{2,i})_{t^*}$. 
In this case, $x \in (\spt \cM^{1,i})_{t^*} \cap N_{t^*}$ yet $x \notin  (\spt \cM^{2,i})_{t^*}$. 
Since $(\spt \cM^{2,i})_t$ is the spacetime support of $\cM^{2,i}_t$, the fact that $x \notin (\spt \cM^{2,i})_{t^*}$ implies that there is a small spacetime neighborhood of $(x,t^*)$ such that $\cM_t$ coincides with $\cM^{1,i}_t$. 
Thus, a tangent flow of $\cM^{1,i}_t$ at $(x,t^*)$ is also a tangent flow of $\cM_t$ at $(x,t^*)$. 
Since any tangent flow of $\cM^{1,i}_t$ at $(x,t^*)$ is a hyperplane coinciding with $\dd \bH$ and since $\cM_t$ has no higher multiplicity planar tangent flows, we find that $\cM^{1,i}_t$ has a multiplicity one planar tangent flow at $(x,t^*)$. 
This implies that $\cM^{1,i}_t$ is smooth and unit density in a backwards spacetime neighborhood of $(x,t^*)$. 
Due to the smoothness of $\cM^{1,i}_t$ and $N_t$ in a neighborhood of $(x,t^*)$ and since $(\spt \cM^{1,i})_t\subseteq K_t$ for all $t \geq t_i$, the strong maximum principle implies that $\cM^{1,i}_t$ coincides with $N_t$ in a backwards spacetime neighborhood of $(x,t^*)$ (see~\cite[Proposition~3.3]{CHHW} or our arguments in the proof of Theorem \ref{theorem main}). 
By standard elliptic theory, $N_t$, $\reg \cM^{1,i}_t$, and $\reg \cM_t$ are spatially real analytic. 
Since $x\notin (\spt \cM^{2,i})_t$, $\reg \cM^{1,i}_t$ coincides with $\reg \cM_t$ in a small spacetime neighborhood of $(x,t^*)$. 
By the identity theorem for real analytic functions, either $N_{t^*} \subseteq \reg \cM_{t^*}$ or a connected component of $\reg \cM_{t^*}$ is a subset of $N_{t^*}$. 
This contradicts our assumption at the beginning of the proof. 
This proves~\eqref{goal Brakke 3} under Case~1.

\textbf{Case 2}: $x \in  (\spt \cM^{2,i})_{t^*}$. 
In this case, $x \in (\spt \cM^{1,i})_{t^*} \cap (\spt \cM^{2,i})_{t^*} \cap N_{t^*}$. 
Since $(\spt \cM^{2,i})_t \subseteq \ov{\bR^{n+1}\setminus K_t}$, we may apply the same reasoning as for $\cM^{1,i}_t$ to find that any tangent flow of $\cM^{2,i}_t$ at $(x,t^*)$ is supported on the hyperplane $\dd\bH$ and has some integer multiplicity. 
A tangent flow of $\cM^{2,i}_t$ at $(x,t^*)$ arises from choosing a convergent subsequence of rescalings of $\cM^{2,i}_t$, based at $(x,t^*)$, by $\la_j \to \infty$. Given $k_2 \in \bN$, suppose that $k_2 \cH^{n} \lfloor \dd \bH$ is a tangent flow of $\cM^{2,i}_t$ at $(x,t^*)$ arising from the sequence of rescalings $\la_j \to \infty$. 
By taking a further subsequence of $\la_{j'}$, we can find a tangent flow of $\cM^{1,i}_t$ at $(x,t^*)$ given by $k_1 \cH^n \lfloor \dd \bH$. Since $\cM_t = \cM^{1,i}_t + \cM^{2,i}_t$, we then find that $(k_1 + k_2)\cH^n \lfloor \dd \bH$ is a tangent flow of $\cM_t$ at $(x,t^*)$ arising from the rescalings $\la_{j'}\to \infty$. 
Since $\cM_t$ has no higher multiplicity planar tangent flows, we find that $k_1 + k_2 = 1$. 
However, since $x \in (\spt \cM^{1,i})_{t^*} \cap (\spt \cM^{2,i})_{t^*}$, we have that $k_1, k_1 \geq 1$. This contradicts $k_1 + k_2 = 1$ and proves~\eqref{goal Brakke 3} under Case 2.

In both cases, we have found a contradiction. 
This means that for each $i$ and each $t > t_i$, $(\spt \cM^{1,i})_t \cap N_t = \emptyset$. 
Similar reasoning holds for $\cM^{2,i}_t$, so we conclude~\eqref{goal Brakke 3}. 
As explained earlier, this implies~\eqref{goal-Brakke 2}.

Given~\eqref{equation spt bound 2} and~\eqref{goal-Brakke 2}, we can now check that the Hausdorff dimension of $(\spt \cM)_t \cap N_t$ is non-increasing for $t \in [0, \infty)$. 
If $t_0 = 0$, then~\eqref{goal-Brakke 2} implies that $\dim((\spt \cM)_t \cap N_t)$ is non-increasing. 
If $t_0 >0$, then by~\eqref{equation spt bound 2} and the definition of $t_0$, we have that $\dim((\spt \cM)_0 \cap N_0) \geq n-1$ and $\dim((\spt \cM)_t\cap N_t)=n-1$ for $t \in (0,t_0)$. 
We then conclude that $\dim((\spt \cM)_t \cap N_t)$ is non-increasing in the case that $t_0>0$.

In the course of our argument, we have shown ~\eqref{goal Brakke 3}. Since there is nothing special about the time $t_i$ aside from the assumption that $\dim((\spt \cM)_{t_i} \cap \dd K_{t_i})< n-1$, we have shown the last statement of~\eqref{item:splittableGAP2}. 
\end{proof}

\begin{rmk}\label{remark:IntPrinciple-LocalizableBF}
    We note that there is a partial converse to Theorem \ref{theorem-splittable-implies-GAP}. 
    If $\cM_t$ satisfies the intersection principle, namely~\eqref{item:splittableGAP2} in Theorem \ref{theorem-splittable-implies-GAP}, then $\cM_t$ is a localizable Brakke flow.
    This can be proved using similar arguments as in the proof of Proposition~\ref{prop-localizable-implies-splittable}.
\end{rmk}

\begin{defn}\label{definition-localizable}
    We say that a level set flow $M_t$ in $\bR^{n+1}$ is \textbf{localizable} if it satisfies the following property: 
    If $K\subseteq \bR^{n+1}$ is a smooth closed domain such that $\dim(\dd K \cap M_{t_0}) < n-1$ for some $t_0\geq 0$, then 
    \begin{enumerate}
        \item \label{loc-LSF-1}
        $M_t 
        = F_{t-t_0}(M_{t_0} \cap K)  
        \cup F_{t-t_0}\pr{M_{t_0} \cap \ovl{\bR^{n+1}\setminus K}}$ for $t\geq t_0$, and
        \item \label{loc-LSF-2}
        $F_{t-t_0}(M_{t_0} \cap K)$ and 
        $F_{t-t_0}\pr{M_{t_0} \cap \ovl{\bR^{n+1}\setminus K}}$ are disjoint for $t> t_0$.
    \end{enumerate}
\end{defn}

One can see that localizable level set flows are closely related to localizable Brakke flows. 
In fact, we will show that there are localizable Brakke flows supported on the inner and out flows of a localizable level set flow.
This will ultimately be applied to non-discrepant level set flows with finitely many singularities in Section~\ref{sec:non-disc-lsf}.

\begin{prop}\label{prop-localizable-implies-splittable}
    Let $M_t$ be a level set flow starting from a smooth closed embedded hypersurface $M_0\subset \bR^{n+1}$. 
    If $M_t$ is localizable, then there are integral, unit regular Brakke flows $\cM^{\rm in}_t$ and $\cM^{\rm out}_t$ which are localizable and are supported on the inner and outer flows $M^{\rm in}_t$ and $M^{\rm out}_t$, respectively.
\end{prop}

\begin{proof}
    The existence of integral, unit regular Brakke flows $\cM^{\rm in}_t$ and $\cM^{\rm out}_t$ which are supported on the inner and outer flows $M^{\rm in}_t$ and $M^{\rm out}_t$ follows from~\cite[Appendix B]{HershkovitsWhite17}. 
    Specifically, ~\cite[Appendix B]{HershkovitsWhite17} gives this result for the outer flow, but the same arguments hold for the inner flow. 

To prove this proposition, we only need to show that $\cM^{\rm in}_t$ and $\cM^{\rm out}_t$ are localizable. 
We will prove this just for $\cM^{\rm out}_t$, and the argument for $\cM^{\rm in}_t$ is exactly the same. 

Suppose that $K$ is a smooth compact domain such that $\dim(\dd K \cap M_{t_0}) < n-1$ for some $t_0\geq 0$. 
We now define $\cM^{\rm{out}, 1}_t$ and $\cM^{\rm{out}, 2}_t$ for $t \geq t_0$ by 
\begin{align*}
\cM^{\rm{out}, 1}_t &:= \cM^{\rm out}_t \lfloor F_{t-t_0}(M_{t_0} \cap K),\text{ and}\\
\cM^{\rm{out}, 2}_t &:= \cM^{\rm out}_t \lfloor F_{t-t_0}(M_{t_0} \cap \ovl{\bR^{n+1}\setminus K}).
\end{align*}
By construction of $\cM^{\rm out}_t$, $(\spt \cM^{\rm out})_t \subseteq M_t$ for each $t \geq t_0$. 
Since $M_t$ is a localizable level set flow, $(\spt \cM^{\rm out})_t \subseteq M_t = F_{t-t_0}(M_{t_0}\cap K) \cup F_{t-t_0}(M_{t_0}\cap \ovl{\bR^{n+1}\setminus K})$, and these sets are disjoint for $t>t_0$.

We have that $\cM^{\rm{out}, 1}_t$ and $\cM^{\rm{out}, 2}_t$ are both Brakke flows for $t > t_0$ since $F_{t-t_0}(M_{t_0}\cap K)$ and $F_{t-t_0}(M_{t_0}\cap \ovl{\bR^{n+1}\setminus K})$ are disjoint. 
Moreover, the disjointness together with $M_t = F_{t-t_0}(M_{t_0}\cap K) \cup F_{t-t_0}(M_{t_0}\cap \ovl{\bR^{n+1}\setminus K})$ implies that 
\begin{align*}
\cM^{\rm out}_t = \cM^{\rm{out}, 1}_t + \cM^{\rm{out}, 2}_t
\end{align*}
for each $t>t_0$.
In order to check that $\cM_t^{\rm{out}, 1}$ is a Brakke flow for $t\geq t_0$, given that it is already a Brakke flow for $t>t_0$, it is enough to check that for each $t \in (t_0, \infty)$ and each $\phi \in C^2_c(\bR^{n+1}, \bR_{\geq 0})$,
\begin{align}\label{eqn:BFtgeqt0}
    \cM_t^{\rm{out}, 1}(\phi) - \cM_{t_0}^{\rm{out}, 1}(\phi) 
    \leq \int_{t_0}^{t} \int 
    \pr{-\phi |\vec{H}|^2 +  \pair{\n\phi, \vec{H}}}
    d \cM_{s}^{\rm{out}, 1} ds.
\end{align}

Let $B_{\varepsilon}(X):=\cup_{x\in X} B_\varepsilon(x)$ denote the open $\varepsilon$-neighborhood of a set $X$. Since $F_{t-t_0}(M_{t_0}\cap K)$ is a weak set flow, it avoids all initially disjoint smooth closed embedded MCFs. This implies that for each $\eps>0$, there exists $\delta>0$ such that $F_{t-t_0}(M_{t_0}\cap K) \subset B_{\varepsilon}(M_{t_0}\cap K)$ for all $0 \leq t -t_0 < \delta$. In particular, 
\begin{align}\label{eqn:epsneighborhoodFt}
F_{t-t_0}(M_{t_0}\cap K)\subset B_{\varepsilon}(K)
\end{align}
for all $0 \leq t-t_0 < \delta$. 
    Based on ~\eqref{eqn:epsneighborhoodFt}, for each $\varepsilon>0,$ 
    if we let $\psi_{\varepsilon}: B_{2\varepsilon}(K) \to \bR_{\geq 0}$ be a smooth cutoff function such that $\psi_{\varepsilon} \equiv 1$ on $B_{\varepsilon}(K)$ and $\psi_{\varepsilon} \equiv 0$ on $\bR^{n+1}\setminus B_{2\varepsilon}(K)$, we have, for each $\phi,$
    \begin{align}
    \lim_{\delta \to 0^+} \cM^{\rm out, 1}_{t_0+\delta}(\phi) 
    = \lim_{\delta \to 0^+} \pr{\cM^{\rm out, 1}_{t_0+\delta}\lfloor B_\varepsilon(K)} (\phi) 
    &\le \lim_{\delta \to 0^+} \pr{\cM^{\rm out}_{t_0+\delta}\lfloor B_\varepsilon(K)} (\phi) \nonumber\\
    & \leq \lim_{\de \to 0^+} \cM^{\rm out}_{t_0+\delta}(\psi_{\varepsilon} \phi)\nonumber\\
    &\label{M1out-approxmassdecreasing}\le \cM^{\rm out}_{t_0}(\psi_{\varepsilon}\phi)
    \end{align}
    where the last inequality uses \cite[7.2(ii)]{Ilmanen94}.
    Since ~\eqref{M1out-approxmassdecreasing} holds for each $\varepsilon>0$, it follows that for each $\phi,$ 
    \begin{align}\label{eqn:massM1}
    \lim_{\delta \to 0^+} \cM^{\rm out,1}_{t_0+\delta}(\phi) 
    \leq (\cM^{\rm out}_{t_0})(\chi_K \phi) 
    = \cM^{\rm out, 1}_{t_0}(\phi)
    \end{align} 
    based on the definition of $\cM^{\rm out,1}$ again, where $\chi_K$ is the characteristic function of $K$ and is the limit of $\psi_\varepsilon$ as $\varepsilon\to 0^+.$

We will now use~\eqref{eqn:massM1} to prove~\eqref{eqn:BFtgeqt0}. Since $\cM^{\rm{out},1}_t$ is a Brakke flow for $t>t_0$, we have that for each $t \in (t_0, \infty)$, $\de>0$, and $\phi \in C^2_c(\bR^{n+1}, \bR_{\geq 0})$,
\begin{align}\label{eqn:BFt>t0}
    \cM_t^{\rm{out}, 1}(\phi) - \cM_{t_0 + \de}^{\rm{out}, 1}(\phi) 
    \leq \int_{t_0 + \delta}^t \int 
    \pr{-\phi |\vec{H}|^2 + \pair{\n\phi, \vec{H}}}
    d \cM_{s}^{\rm{out}, 1} ds.
\end{align}
Combining~\eqref{eqn:massM1} with~\eqref{eqn:BFt>t0},
\begin{align*}
    \cM^{\rm{out}, 1}_t(\phi) - \cM^{\rm{out},1}_{t_0}(\phi) &\leq \lim_{\de \to 0^+} \left(\cM^{\rm{out}, 1}_t(\phi) - \cM^{\rm{out},1}_{t_0+\de}(\phi)\right)\\
    &\leq \lim_{\de \to 0^+} \int_{t_0 + \delta}^t \int \pr{-\phi |\vec{H}|^2 + \pair{\n\phi, \vec{H}}} d \cM_{s}^{\rm{out}, 1} ds\\
    &\leq \int_{t_0}^{t} \int \pr{-\phi |\vec{H}|^2 + \pair{\n\phi, \vec{H}}} d \cM_{s}^{\rm{out}, 1} ds,
\end{align*}
where the last line follows since $\cM^{\rm{out}}_t$ is a Brakke flow, so $|\vec{H}|^2$ and $\pair{\n\phi, \vec{H}}$ are both locally integrable in spacetime with respect to $\cM^{\rm{out}, 1}_t$.
We conclude that $\cM^{\rm{out},1}_t$ is a Brakke flow for $t \geq t_0$, and the same argument shows that $\cM^{\rm{out},2}_t$ is also a Brakke flow for $t \geq t_0$.
This verifies \eqref{loc-BF-1} in Definition \ref{definition-splittable}.

Since $\dim(\dd K \cap M_{t_0})< n-1$, we have that $\dim(\dd K \cap (\spt \cM^{\rm out})_{t_0}) < n-1$. 
This set is negligible for an $n$-dimensional Brakke flow, so $\cM^{\rm out}_{t_0} = \cM^{\rm{out},1}_{t_0} + \cM^{\rm{out},2}_{t_0}$. 
Combined with~\eqref{eqn:BFtgeqt0}, we find that \eqref{loc-BF-2} in Definition \ref{definition-splittable} holds for $\cM_t^{\rm in}$ and $\cM_{t}^{\rm out}$. This shows that $\cM_t^{\rm in}$ and $\cM_{t}^{\rm out}$ are localizable, which concludes the proof of the proposition.
\end{proof}

\begin{rmk}\label{remark:ConverseToLocalizable-Splittable}
    We remark that there is a partial converse to Proposition~\ref{prop-localizable-implies-splittable}. 
    Suppose $M_t$ is a level set flow starting from a smooth closed embedded hypersurface, and suppose that the integral, unit regular Brakke flows $\cM_t^{\rm in}, \cM_t^{\rm out}$ from~\cite[Appendix B]{HershkovitsWhite17} are localizable, have no higher multiplicity planar tangent flows, and satisfy the following condition: for any smooth MCF $N_t$ and each $t \in [0, \infty)$, $N_t \not\subseteq \reg \cM_t$ and no connected component of $\reg \cM_t$ is a subset of $N_t$. 
    Then, $M_t$ is a localizable level set flow. 
\end{rmk}

Combining Theorem~\ref{theorem-splittable-implies-GAP} and Proposition~\ref{prop-localizable-implies-splittable}, we obtain that a localizable level set flow, subject to a multiplicity assumption, has monotone intersection dimension with smooth flows. 

\begin{thm}\label{thm:loc-LSF-mono}
    Let $M_t$ be a compact, non-fattening level set flow starting from a smooth closed embedded hypersurface $M_0$. 
    Suppose that $M_t$ is localizable and the inner and outer Brakke flows $\cM_t^{\rm in}$, $\cM_t^{\rm out}$ have no higher multiplicity planar tangent flows.    
     Then, for each smooth closed connected mean curvature flow of hypersurfaces $N_t$, one of the following conditions holds:
    \begin{enumerate}
        \item The Hausdorff dimension of $M_t \cap N_t$ is non-increasing for as long as $N_t$ exists. 
        Moreover, if $\dim(M_{t_0}\cap N_{t_0})< n-1$, then $M_s \cap N_s = \emptyset$ for all $s >t_0$.  
        \item For some $t \in [0, \infty)$, either $N_t$ is a subset of $M^{\rm in}_t$ or $M^{\rm out}_t$, or a connected component of $M^{\rm in}_t$ or $M^{\rm out}_t$ is a subset of $N_t$. 
    \end{enumerate}    
\end{thm}

\begin{proof}
By Proposition \ref{prop-localizable-implies-splittable}, $\cM^{\rm out}_t$ and $\cM^{\rm in}_t$ are localizable. 
Since these Brakke flows have no higher multiplicity planar tangent flows by assumption, Theorem~\ref{theorem-splittable-implies-GAP} applies to both $\cM^{\rm out}_t$ and $\cM^{\rm in}_t$. 
Since $\cM^{\rm out}_t$ and $\cM^{\rm in}_t$ are supported on $M^{\rm out}_t$ and $M^{\rm in}_t$, respectively, and since $M_t = M^{\rm out}_t \cup M^{\rm in}_t$ by Lemma~\ref{lem:nonfat-in+out}, we conclude this theorem.
\end{proof}

\begin{rmk}\label{remark:ConverseToLocLSF-IntPrinciple}
    We note that there is a partial converse to Theorem \ref{thm:loc-LSF-mono}, which follows from combining Remarks~\ref{remark:IntPrinciple-LocalizableBF} and \ref{remark:ConverseToLocalizable-Splittable}. The idea is that under reasonable assumptions, a level set flow is localizable if and only if its inner and outer flows satisfy the intersection principle.
    
    Suppose that $M_t$ is a level set flow starting from a smooth closed embedded hypersurface, and let $\cM_t^{\rm in}$, $\cM_t^{\rm out}$ be the unit regular, integral Brakke flows associated to the inner and outer flows of $M_t$ from~\cite[Appendix B]{HershkovitsWhite17}. Suppose that $\cM_t^{\rm in}, \cM_t^{\rm out}$ have no higher multiplicity planar tangent flows, and suppose that for any smooth MCF $N_t$ and each $t\in [0, \infty)$, $N_t \not\subseteq \reg\cM_t$ and no connected component of $\reg \cM_t$ is a subset of $N_t$. If $\cM_t^{\rm in}, \cM_t^{\rm out}$ both satisfy the intersection principle, namely~\eqref{item:splittableGAP2} from Theorem \ref{theorem-splittable-implies-GAP}, then $M_t$ is localizable.
\end{rmk}

\begin{rmk}\label{rmk-removing-high-mul}
    We will see in the next section that, as a corollary of the statements above, we attain the analogue of Theorem \ref{theorem main} for level set flows with finitely many singularities. 
    In fact, we may drop the condition of no higher multiplicity planar tangent flows when working in the situation of finitely many singularities. 
    This is because we do not need the full power of the maximum principle in \cite{CHHW} and the parabolic stratification in \cite{Wh97} when there are only finitely many singularities.
    See the proof of Theorem~\ref{thm:dim-mono-nonfattening} at the end of Section~\ref{sec:non-fat-lsf}.
\end{rmk}


\subsection{Non-discrepant level set flows}
\label{sec:non-disc-lsf}

In this section, we work with compact non-discrepant level set flows. 
The main result in this section is the following proposition.
It proves the localizability of non-discrepant level set flows with finitely many singularities.
As a consequence, we can apply Theorem~\ref{thm:loc-LSF-mono} to such flows.

\begin{prop}
\label{prop:localized-LSF}
    Let $M_t$ be a compact non-discrepant\footnote{After proving Theorem~\ref{thm:finite-S-nonfattening}, we know that this proposition is also true for non-fattening level set flows. See Remark~\ref{rmk:LSF-localize-nonfattening}.} level set flow with finitely many singularities.
    Suppose that at time $t_0,$ $M_{t_0}$ is connected and has singularities at
    $x_1,\cdots, x_k\in\bb R^{n+1}.$
    If $M_{t}$ has $\ell$ smooth connected components $M_t^1,\cdots, M_t^\ell$ for $t\in (t_0,t_0+\delta)$ 
    for some $\delta>0,$ then we can write $M_{t_0} = M^1\cup \cdots \cup M^\ell$ such that
\begin{enumerate}
     \item\label{local-LSF-1} 
     $M^i\cap M^j\sbst\set{x_1,\cdots,x_k}$ for $i\neq j,$ 
     \item\label{local-LSF-conn}
     $\bb R^{n+1}\setminus M^i$ has exactly two components for each $i,$ and
     \item\label{local-LSF-2} 
     $F_s\pr{M^i} = M^i_{t_0+s}$ for $s\in (0,\delta).$
\end{enumerate}
    In particular, a non-discrepant level set flow with finitely many singularities is localizable.
\end{prop}

To prove Proposition~\ref{prop:localized-LSF}, we need a consequence of White's result on possible topological changes of a level set flow~\cite[Theorems~1(i) and 5.2]{White95}.
It implies the following proposition, in particular, when there are only finitely many singularities.

\begin{prop}[\!\mbox{\cite[Theorem~1(i)]{White95}}]
\label{prop:component}
    Let $M_t$ be a compact level set flow in $\bb R^{n+1}$ and suppose $M_t$ has only spacetime isolated singularities. 
    If $\bb R^{n+1}\setminus M_0$ has $\ell$ components for some $\ell\in\bb N$, then there is $\delta>0$ such that $\bb R^{n+1}\setminus M_t$ has $\ell$ components for $t\in[0,\delta]$.
\end{prop}

In \cite{White95}, White proved more general results than the proposition above.  
Here, we only use the part about the number of components of the complement of the flow.

\begin{proof}
[Proof of Proposition~\ref{prop:localized-LSF}]
    Since $M_t$ is a non-discrepant level set flow with finitely many singularities, it supports a unit regular Brakke flow $\cM_t$ which is the boundary motion Brakke flow associated to $M_t.$ Moreover, it satisfies $M_t = \spt \cM_t$ unless an entire component goes extinct at a singular point (see~\cite[Proposition~3.4]{PayneMassDrop}).

    Let $\cS:=\set{x_1,\cdots,x_k}.$
    By our assumption and definition of singularities of level set flow (Definition \ref{def:lsf-sing}), $M_t\setminus \cS$ converges to $M_{t_0}\setminus\cS$ in $C^\infty_{\rm loc}\pr{\bb R^{n+1}\setminus\cS}$ as $t\searrow t_0.$
    Thus, for each $i=1,\cdots, \ell,$ we define $\mathring M^i$ to be the limit of $M^i_t$ in $C^\infty_{\rm loc}\pr{\bb R^{n+1}\setminus\cS}$ as $t\to t_0^+$ and let $M^i$ be the closure of $\mathring M^i.$
    Since $M_t = M^1_t\cup\cdots\cup M^\ell_t$ for $t>t_0,$ this construction implies $M_{t_0} = M^1\cup\cdots\cup M^\ell.$

    We first show \eqref{local-LSF-1}.
    Assume $i=1,$ $j=2,$ and $x_0\in M^1\cap M^2\setminus\cS.$
    This means that $x_0$ is a regular point of $M_{t_0}$ and we can write both $M^1_t$ and $M^2_t$ as graphs over $M_{t_0}$ locally near $x_0.$
    To be precise, there exists $\varepsilon>0$ and $r>0$ such that for $t-t_0\in (0,\varepsilon),$ we can find smooth functions $f^1_t, f^2_t \colon T_{x_0}M_{t_0}\cap B_r\to\bb R$ such that 
    \begin{itemize}
	   \item the connected component of
	$M^1_t \cap B_{r}(x_0)$ containing $x_0$ is the graph of $f^1_t$ and 
	   \item the connected component of
	$M^2_t \cap B_{r}(x_0)$ containing $x_0$ is the graph of $f^2_t$
    \end{itemize}
    so that $f^1_t$ and $f^2_t$ converge to $f^1$ and $f^2$ in $C^\infty\pr{T_{x_0}M_{t_0}\cap B_r}$ and the connected component of $M^i\cap B_r(x_0)$ containing $x_0$ is the graph of $f^i$ for $i=1,2.$
    By the construction of $M^i$'s, we have
    \begin{align*}
        \pr{M^1\cup M^2}\cap B_r(x_0)
        \sbst M_{t_0}\cap B_r(x_0).
    \end{align*}
    However, this implies a tangent flow $\cT_t$ of $\cM_{t_0}$ at $x_0$ is a high multiplicity plane for $t>0,$ contradicting the regularity of $\cM_{t_0}$ at $x_0$ and the unit regularity of $\cM_t.$
    The same argument works for any $i \neq j$, so this finishes the proof of~\eqref{local-LSF-1}.

    We present the proof for \eqref{local-LSF-2} when $k=1$ and $\ell=2.$
    The same arguments work for general $k$ and $\ell.$
    By a translation, we may assume $(x_1, t_0)=(0,0).$
    We let $M^i_0:= M^i$ so that $M^i_t$'s are defined for $t\in [0,\delta)$ and $i=1,2.$

    We are going to prove $F_t\pr{M^1} = M_t^1$ for $t\in (0,\delta),$ where we are using that $t_0=0$. The same arguments work for $M^2$ as well.
    First, we know that 
    \begin{align}\label{lsf-cond-1}
        F_t\pr{M^1}
        \sbst F_t(M_0)
        = M_t
        = M^1_t\cup M^2_t
    \end{align}
    based on Lemma~\ref{lem:Ilm-contain}.
    We will show that
    \begin{align}\label{lsf-claim-1}
        M^1_t\text{ is a weak set flow},
    \end{align}
    and that
    \begin{align}\label{lsf-claim-2}
        F_t\pr{M^1}\cap M^2_t=\emptyset
        \text{ for }t>0.
    \end{align}
    By definition of level set flow, $F_t\pr{M^1}$ is the maximal weak set flow, so~\eqref{lsf-claim-1} implies $M^1_t \subseteq F_t\pr{M^1}$. 
    Combining~\eqref{lsf-cond-1} and~\eqref{lsf-claim-2} with the assumption that $M_t$ splits into the two connected components $M^1_t$ and $M^2_t$, we get the conclusion \eqref{local-LSF-2}  that $F_t\pr{M^1} = M^1_t$.

    We prove \eqref{lsf-claim-1} first. Recall that a weak set flow is a flow of closed sets which satisfies the avoidance principle with respect to smooth MCF.
    Since $M^1_t$ is a smooth MCF for $t\in(0,\delta),$ it suffices to check that $M^1_t$ avoids any smooth closed MCF starting from $t=0.$
    Let $N_t$, for $t\in[0,T_2]\sbst[0,\delta)$, be a smooth closed mean curvature flow with
    \begin{align}\label{N0-disjoint}
        N_0\cap M^1_0=\emptyset.
    \end{align}
    In particular, since $x_0=0\in M^1_0,$ we can take $r > 0$ such that
    \begin{align}\label{N0-r-disjoint}
        N_0\cap B_{r}=\emptyset,
    \end{align}
    where $B_r$ is a ball around $0$. The sequence of hypersurfaces $M^1_t\setminus B_{r/2}$ converges to $M^1_0\setminus B_{r/2}$ as $t\to 0^+$ in $C^\infty_{\rm loc}\pr{\bb R^{n+1}\setminus B_{r/2}}.$
    Thus, \eqref{N0-disjoint} implies that for some $T>0,$ we have
    \begin{align}\label{NM-disjoint-1}
        N_t\cap M^1_t\setminus B_{r/2}=\emptyset
    \end{align}
    for $t<T.$
    On the other hand, the classical avoidance principle and \eqref{N0-r-disjoint} imply 
    \begin{align}\label{NM-disjoint-2}
        N_t\cap B_{\sqrt{r^2-2nt}}
        =\emptyset
    \end{align}
    for $t < r^2/2n.$
    Combining \eqref{NM-disjoint-1} and \eqref{NM-disjoint-2}, we get
    \begin{align*}
        N_t\cap M_t^1
        & = \pr{N_t\cap M_t^1\setminus B_{\sqrt{r^2-2nt}}}
        \cup \pr{N_t\cap M^1_t\cap B_{\sqrt{r^2-2nt}}}\\
        & \sbst \pr{N_t\cap M_t^1\setminus B_{r/2}}
        \cup \pr{N_t\cap B_{\sqrt{r^2-2nt}}}
        = \emptyset
    \end{align*}
    for time $t < \min\set{T, 3r^2/8n}.$
    It then follows that $N_t\cap M_t^1=\emptyset$ for all $t\in (0,T_2]$ by the classical avoidance principle since these flows are immediately smooth for short time after time $0$.
    This proves that $M^1_t$ is a weak set flow.

    Next, we prove \eqref{lsf-claim-2}. Recall that if $X$ is a smooth, connected, embedded hypersurface with nonempty boundary, then its level set flow $X_t$ vanishes immediately, i.e.\ $X_t = \emptyset$ for $t>0$~\cite[Theorem~8.1]{EvansSpruck91}.
    Now, if \eqref{lsf-claim-2} were not true, then by \cite[Theorem~8.1]{EvansSpruck91} and the finiteness of the singular sets of the flows, there would exist $\ovl T>0$ such that
    \begin{align}\label{lsf-ctrd-1}
        F_t\pr{M^1_0} = M^1_t\cup M^2_t
    \end{align}
    for $t\in (0, \ovl T).$
    Fix a point $x_0\in M_0^2\setminus\set 0$ and shrink $r>0$ so small that 
    \begin{align}\label{lsf-ctrd-2}
        B_r(x_0)\cap M^1_0 = \emptyset.
    \end{align}
    Since $M^2_0$ is regular near $x_0,$ $\dd B_r(x_0)$ intersects $M^2_0$ transversely after we shrink $r>0$ if necessary.
    This implies 
    \begin{align}\label{lsf-ctrd-3}
        \dd B_{\sqrt{r^2-2nt}}(x_0)\cap M^2_t\neq \emptyset
    \end{align}
    for $t\in (0,\ovl T)$ if we shrink $\ovl T.$ 
    By \eqref{lsf-ctrd-2} and the fact that $F_t(M^1_0)$ is a level set flow, we have
    \begin{align*}
        F_t\pr{M^1_0}\cap \dd B_{\sqrt{r^2-2nt}}(x_0)=\emptyset
    \end{align*}
    for $t\le r^2/2n.$
    This contradicts \eqref{lsf-ctrd-3} based on \eqref{lsf-ctrd-1}.
    Hence, we prove $F_t\pr{M^1} = M_t^1$ for $t\in (0,\delta)$ and then \eqref{local-LSF-2} follows.

    We next prove \eqref{local-LSF-conn}, and we recall that we've assumed $t_0=0.$
    By \eqref{local-LSF-2}, each $M^i_t$ is a level set flow for $t\in [0,\delta].$
    Since $M^i_t$ is a smooth connected closed hypersurface for $t\in (0,\delta]$, its complement has two components.
    Thus, \eqref{local-LSF-conn} follows from Proposition~\ref{prop:component}.
    This finishes the proof of the first part of Proposition~\ref{prop:localized-LSF}.

    Now, we show that $M_t$ is localizable if it only has finitely many singularities.
    It suffices to check the localizability condition at each singular time.
    We assume that at a singular time $t_0=0,$ $M_t$ satisfy the description in the first part of Proposition~\ref{prop:localized-LSF}.
    Let $K$ be a smooth closed domain in $\bb R^{n+1}$ such that $\dim(\dd K\cap M_{0})<n-1.$
    In particular, for any $i=1,\cdots,\ell,$ $\cH^{n-1}\pr{M^i\cap \dd K}=0.$
    Thus, by Lemma~\ref{lem:low-dim-one-side}, either $M^i\sbst K$ or $M^i\sbst \ovl{\bb R^{n+1}\setminus K}.$
    We collect those $i$'s such that $M^i\sbst K$ to form a set $I$ and collect those $j$'s such that $M^j\sbst \ovl{\bb R^{n+1}\setminus K}$ to form a set $J.$
    Note that $I\cap J=\emptyset$ since for each $i,$ exactly one of the situations happens because of the dimension assumption.
    Then \eqref{local-LSF-2} and the fact that the level set flows of disjoint compact sets are disjoint and independent \cite[2.1~(2)]{White00} imply
    \begin{align*}
    M_{s}
    & = F_s\pr{\bigcup_{i\in I} M^i}
    \cup F_s\pr{\bigcup_{j\in J} M^j}\\
    & = F_s\pr{M_{0}\cap K}
    \cup F_s\pr{M_{0}\cap \ovl{\bb R^{n+1}\setminus K}}.
    \end{align*}
    The fact that $M_t$ has exactly $\ell$ components for $t\in(0, \delta)$ means that $F_s\pr{M_{0}\cap \ovl{\bb R^{n+1}\setminus K}}$ and $F_s\pr{M_{0}\cap K}$ are disjoint for $s>0.$
    These prove that $M_t$ is a localizable level set flow, and finish the proof of Proposition~\ref{prop:localized-LSF}.
\end{proof}

Combining Proposition~\ref{prop:localized-LSF} and Theorem~\ref{thm:loc-LSF-mono} gives a dimension monotonicity result for non-discrepant level set flows with finitely many singularities which do not have high multiplicity planar tangent flows. The condition of no higher multiplicity planar tangent flows can be dropped in the case of finitely many singularities.
As we will ultimately get a result for non-fattening flows, we leave the proof to Section~\ref{sec:non-fat-lsf}
(see Remark~\ref{rmk-removing-high-mul}).


\subsection{Non-fattening level set flows}\label{sec:non-fat-lsf}

In this section, we will prove the equivalence between non-fattening and non-discrepancy under the assumption of finitely many singularities.
As a consequence, we will prove Theorem~\ref{thm:dim-mono-nonfattening}.

\begin{thm}
\label{thm:finite-S-nonfattening}
    Let $M$ be a smooth, closed, embedded hypersurface in $\bb R^{n+1}$ and $M_t$ be the level set flow starting from~$M.$
    Suppose $M_t$ has only finitely many singularities.
    If $M_t$ is non-fattening, then $M_t$ is non-discrepant.
\end{thm}

\begin{proof}
    Suppose for a contradiction that $T=T_{\rm disc}<\infty,$ where
    \begin{align*}
        T_{\rm disc}
        := \inf\set{t\ge 0: M_t, M^{\rm out}_t,\text{ and }M^{\rm in}_t \text{ are not the same}
        }
    \end{align*}
    is the discrepancy time defined in \cite{HershkovitsWhite17}.
    Note that $T$ must be a singular time and hence an isolated singular time by the finiteness assumption.
    
    We let $M=\dd D$ where $D$ is the compact domain bounded by $M.$
    By \cite[Theorem~B.2]{HershkovitsWhite17}, we have
    \begin{align}\label{M-in-out-T}
        M^{\rm out}_T
        = \lim_{t\nearrow T} \dd F_t(D)
        = \lim_{t\nearrow T} \dd F_t\pr{\ovl{\bb R^{n+1}\setminus D}}
        = M^{\rm in}_T
    \end{align}
    where the limits are understood in the Hausdorff sense and the second equality is based on 
    $$\dd F_t(D) = M_t = \dd F_t\pr{\ovl{\bb R^{n+1}\setminus D}}$$ 
    for a regular time $t<T.$
    In particular, by Lemma~\ref{lem:nonfat-in+out}, we have $M_T = M^{\rm out}_T = M^{\rm in}_T.$
    
    Let $\cS:=\set{x\in\bb R^{n+1}: (x,T)\text{ is a singularity of }M_t}$
    be the set of the singularities of $M_t$ at time $T.$ 
    The finiteness assumption again implies that there is $\tau>0$ such that $M_t$ is smooth for $t \in (T, T+\tau]$. 
    By definition of $T = T_{\rm disc}$, there must be $\tilde{\tau} \in (T, T+\tau)$ such that $M^{\rm out}_{\tilde{\tau}} \neq M^{\rm in}_{\tilde{\tau}}$. 
    We now find $\delta>0$ such that 
    \begin{align}\label{disc-T+tau}
        M^{\rm out}_{T+\tilde{\tau}} \setminus B_\delta\pr{\cS} 
        \neq M^{\rm in}_{T+\tilde{\tau}} \setminus B_\delta\pr{\cS},
    \end{align}
    where $B_\delta\pr{\cS}:=\cup_{x\in \cS} B_\delta(x)$ is the $\delta$-neighborhood of the set $\cS.$ 
    Indeed, if~\eqref{disc-T+tau} were not true for all $\de>0$, then we would have that 
    $M^{\rm out}_{T+\tilde{\tau}} \setminus B_\delta\pr{\cS} 
        = M^{\rm in}_{T+\tilde{\tau}} \setminus B_\delta\pr{\cS}$ 
    for all $\de>0$. 
    Letting $\de \to 0$ and using that the singular set is finite, we would find that $M^{\rm out}_{T+\tilde{\tau}} = M^{\rm in}_{T+\tilde{\tau}}$ which contradicts our choice of $\tilde{\tau}$. 
    Thus, we can indeed find $\de>0$ satisfying~\eqref{disc-T+tau}.
        
    We claim that~\eqref{disc-T+tau} implies
    \begin{align}\label{disc-T-to-T+tau}
        M^{\rm out}_{t} \setminus B_\delta\pr{\cS} 
        \neq M^{\rm in}_{t} \setminus B_\delta\pr{\cS}
        \text{ for all }
        t\in (T,T+\tilde{\tau}].
    \end{align}
    Suppose for a contradiction that 
    $M^{\rm out}_{t_0} \setminus B_\delta\pr{\cS} 
    = M^{\rm in}_{t_0} \setminus B_\delta\pr{\cS}$
    for some $t_0\in (T,T+\tilde{\tau}].$
    Thus, real analyticity implies that components of $M^{\rm out}_{t_0}$ and components of $M^{\rm in}_{t_0}$ that are not contained in $B_\delta(\cS)$ all coincide.
    Therefore, their distinct components (if any) are all contained in $B_\delta(\cS).$
    The classical avoidance principle then implies that these components are contained in $B_\delta(\cS)$ for all $t\in [t_0, T+\tilde{\tau}].$
    In particular, we conclude
    \begin{align*}
        M^{\rm out}_{T+\tilde{\tau}} \setminus B_\delta\pr{\cS} 
        = M^{\rm in}_{T+\tilde{\tau}} \setminus B_\delta\pr{\cS},
    \end{align*}
    which contradicts \eqref{disc-T+tau}.
    Thus, we have proven \eqref{disc-T-to-T+tau}.

    Now, we note that the definition of regular points implies 
    \begin{align}\label{u-v-smooth-conv}
        \pr{M^{\rm out}_t \setminus \cS}
        \text{ and }
        \pr{M^{\rm in}_t \setminus \cS}
        \text{ converge to }
        \pr{M_T \setminus \cS}
        \text{ in }C^\infty_{\rm loc}\pr{\bb R^{n+1}\setminus\cS}
        \text{ as }t\searrow T.
    \end{align}
    Thus, for each $x\in M_T\setminus\cS,$ we can find $r_x>0$ such that $M_T\cap B_{r_x}(x)$ is smooth and connected and we can write
    \begin{align*}
        M_T\cap B_{r_x}(x)
        & = {\rm graph}~f_x,\\
        M^{\rm in}_t\cap B_{r_x}(x)
        &= {\rm graph} ~u_x(\cdot,t), \text{ and}\\
        M^{\rm out}_t\cap B_{r_x}(x)
        &= {\rm graph} ~v_x(\cdot,t)
    \end{align*}
    for $t\in[T,T+r_x^2]$ for some functions $f_x\colon \pr{T_x M_T\cap B_{r_x}(x)}\to\bb R$ and $u_x, v_x\colon \pr{T_x M_T\cap B_{r_x}(x)}\times [T,T+r_x^2]\to\bb R$ such that 
    \begin{align}\label{u,v-converge-to-f}
    \lim_{t\searrow T} u_x(\cdot,t)
    = \lim_{t\searrow T} v_x(\cdot,t)
    = f_x.
    \end{align}
    By the compactness of $M_T\setminus B_\delta(\cS),$ we can find $x_1,\cdots, x_\ell\in M_T\setminus\cS$ such that
    \begin{align}\label{M_t-S-cover}
        M_T\setminus B_\delta(\cS)
        \sbst \bigcup_{j=1}^\ell B_{r_{x_j}}(x_j).
    \end{align}
    We shrink $\tau$ and $\delta$ so that $\tau\le \min_j\set{r_{x_j}^2:j=1,\cdots,\ell}$ and $\delta\le \min_j\set{r_{x_j}:j=1,\cdots,\ell}.$
    Based on~\eqref{disc-T-to-T+tau} and~\eqref{M_t-S-cover}, we know that there exists $i_0\in\set{1,\cdots,\ell}$ such that $u_{x_{i_0}}\neq v_{x_{i_0}}.$

	We now construct a flow which collects points ``between'' the inner and outer flows away from the singular set.
	Let $r_i:=r_{x_i},$ 
	$P_i:= T_{x_i}M_T,$ 
	$u_i:=u_{x_i},$ and 
	$v_i:=v_{x_i}.$
	For $t\in (T,T+\tau),$ let $\td M_t$ be the union of $M^{\rm in}_t,$ $M^{\rm out}_t,$ and an additional set $M^{\rm fat}_t,$ which is defined by 
	\begin{align*}
	\set{
	p + w \N_{P_i}(p)\in \bb R^{n+1}\setminus B_\delta(\cS): 
	i\in\set{1,\cdots,\ell},
	p\in P_i\cap \ovl B_{r_i}(x_i),
	\text{ and }
	w\text{ between }
	u_{i}(p,t)\text{ and }v_{i}(p,t)
	}.
	\end{align*}	
	We define $\td M_T:=M_T,$ and we claim that $\{\td M_t\}_{t \in [T,T+\tau)}$ is a weak set flow starting from $M_T.$
	
	Let $\{N_t\}_{t \in [T_1, T_2]}$ be a smooth, closed, connected MCF with $[T_1,T_2]\sbst [T,T+\tau)$ and $N_{T_1}\cap \td M_{T_1}=\emptyset.$
	Suppose for a contradiction that $N_{t_0}\cap \td M_{t_0}\neq\emptyset$ for some $t_0\in (T_1,T_2).$
	Since $M^{\rm out}_t$ and $M^{\rm in}_t$ are weak set flows with
	$M^{\rm out}_T
	= M^{\rm in}_T
	= M_T
	= \td M_T$
	by \eqref{M-in-out-T} and Lemma~\ref{lem:nonfat-in+out},
	they do not intersect $N_t$ for all $t\in [T_1, T_2].$ 
	Thus, we have $N_{t_0}\cap M^{\rm fat}_{t_0}\neq\emptyset.$ Moreover, since $M^{\rm out}_t$ and $M^{\rm in}_t$ are smooth and closed for $t \in (T, T+\tau)$ and since $N_t$ does not intersect them, $N_t$ is a subset of one of the connected components of $\bR^{n+1}\setminus \big(M^{\rm out}_t \cup M^{\rm in}_t\big)$.
    
    For $t\in [T_1, t_0),$ we know that $N_t$ cannot be fully contained in $B_\delta(\cS);$ otherwise, by the classical avoidance principle, $N_t$ would be contained in $B_\delta(\cS)$ for all later $t,$ and then $N_{t_0}\cap M^{\rm fat}_{t_0}=\emptyset,$ a contradiction.
    On the other hand, by the construction of $M^{\rm fat}_t,$ we know that if a connected component $X$ of $\bb R^{n+1}\setminus\pr{M^{\rm out}_t\cup M^{\rm in}_t}$ intersects $M^{\rm fat}_t,$ then
    $X\setminus M^{\rm fat}_t\sbst B_\delta(\cS).$
    Thus, the fact that $N_t\not\sbst B_\delta(\cS)$ implies
	\begin{align}\label{N-int-M-fat}
	N_{t}\cap M^{\rm fat}_{t}\neq\emptyset
    \text{ for all }
    t\in(T,t_0).
	\end{align} 
	We now deal with two cases.
	
	\noindent \textbf{Case 1:}
	$T_1>T.$
	In this case, \eqref{N-int-M-fat} implies $N_{T_1}\cap M_{T_1}^{\rm fat}\neq\emptyset.$
	This violates $N_{T_1}\cap \td M_{T_1}=\emptyset.$
	
	\noindent \textbf{Case 2:}
	$T_1=T.$
	In this case, by the smooth convergence \eqref{u-v-smooth-conv}, \eqref{u,v-converge-to-f}, the smoothness of $N_t,$ and the property of non-trivial intersection \eqref{N-int-M-fat}, we have
	\begin{align*}
	N_T\cap M_T\setminus B_\delta(\cS)
	= \lim_{t\searrow T}
	\pr{N_t\cap \ovl{M^{\rm fat}_t}}
	\neq\emptyset.
	\end{align*}
	This also violates $N_{T_1}\cap \td M_{T_1}=\emptyset.$
	
	In either case, we show that $\td M_t$ is a weak set flow starting from $M_T.$
	In particular, $\td M_t\sbst M_t$ for $t\in(T,T+\tau).$
	However, $\td M_t$ is fattening since $u_{i_0}\neq v_{i_0}.$
	This contradicts the non-fattening assumption, and we conclude that $T_{\rm disc}=\infty$.
\end{proof}

\begin{rmk}\label{rmk:LSF-localize-nonfattening}
    Based on Theorem~\ref{thm:finite-S-nonfattening}, we can upgrade Proposition~\ref{prop:localized-LSF} to non-fattening level set flows.
    In particular, a non-fattening level set flow with finitely many singularities is localizable.
\end{rmk}

Now we can put everything together and prove Theorem~\ref{thm:dim-mono-nonfattening}.
Note that condition \eqref{main-item-2} means that $M_t$ is non-discrepant.

\begin{proof}
    [Proof of Theorem~\ref{thm:dim-mono-nonfattening}]
    Based on Theorem~\ref{thm:finite-S-nonfattening}, \eqref{main-item-1} and \eqref{main-item-2} are equivalent.
    We now prove \eqref{main-item-3}$\Rightarrow$\eqref{main-item-1}, which is simpler than the converse and holds in a more general setting.
    Suppose for a contradiction that $M_t$ is fattening.
    Since it starts from a closed smooth hypersurface $M_0,$ $M_t$ is smooth for $t\in [0,\delta]$ for some $\delta>0.$
    Since it is fattening, there exists $T>0$ such that $\cH^{n+1}\pr{M_T}>0.$
    In particular, we can find $x_0\in\bb R^{n+1}$ and $r_0>0$ such that
    \begin{align}\label{M-fat-1}
        \dim \pr{M_T\cap \dd B_{r_0}(x_0)}
        = n.
    \end{align}
    We now consider a smooth spherical flow $N_t := \dd B_{\sqrt{r_0^2+2n(T-t)}}(x_0),$
    which satisfies $N_T=\dd B_{r_0}(x_0).$
    Theorem~\ref{theorem main} and \eqref{M-fat-1} then imply
    \begin{align*}
        \dim\pr{M_\delta\cap N_\delta}
        \le n-1
        <n = \dim\pr{M_T\cap N_T}.
    \end{align*}
    We may assume $T>\delta$ by shrinking $\delta.$ 
    This contradicts the assumption, \eqref{main-item-3}, that $M_t$ satisfies the intersection principle. Thus, we find that $M_t$ must be nonfattening, which shows that \eqref{main-item-1} holds.

    It remains to prove \eqref{main-item-1} implies \eqref{main-item-3}.
    Combining Theorem~\ref{thm:loc-LSF-mono} and Proposition~\ref{prop:localized-LSF} proves \eqref{main-item-1}$\Rightarrow$\eqref{main-item-3} with an additional assumption that there are no higher multiplicity planar tangent flows.
    This assumption is used in two places in the proof of Theorem~\ref{theorem-splittable-implies-GAP}, which implies Theorem~\ref{thm:loc-LSF-mono}.
    In the rest of the proof, we explain why this multiplicity assumption is not needed in the case of finitely many singularities.

    First, the multiplicity assumption is used in \eqref{equation sing bound 2} to get a dimension bound on the singular set.
    This is automatically true when the singular set is finite.

    Second, the multiplicity assumption is used in the proof of \eqref{goal Brakke 3}.
    There, it is used to get the smoothness of the flows $\cM^{1,i}_t$ at $(x,t^*)$ for some $t^*>t_i$ that contradicts \eqref{goal Brakke 3}.
    In the case of finitely many singularities, we can find $\delta>0$ such that both $\cM_t^{1,i}$ and $\cM_t^{2,i}$ are globally smooth at time $t\in (t_0,t_0+\delta).$
    In particular, they are smooth at time $t^*$ which is close enough to $t_i$ for $i$ large enough.
    This then allows us to apply the strong maximum principle for smooth flows in~\cite[Proposition~3.3]{CHHW} in the proof of  Case 1. Similarly, the strong maximum principle for smooth flows can obviate the need for the multiplicity assumption in the proof of Case 2.
    
    In conclusion, the proof of Theorem~\ref{theorem-splittable-implies-GAP} and hence that of Theorem~\ref{thm:loc-LSF-mono} work in the case of finitely many singularities without the assumption of no higher multiplicity planar tangent flows. This shows that \eqref{main-item-1} implies \eqref{main-item-3} as claimed and concludes the proof of Theorem~\ref{thm:dim-mono-nonfattening}.
\end{proof}


\section{Applications and Examples}\label{sec:ApplicationsandExamples}

\subsection{Applications}
\label{sec:App}

In this section, we collect some applications of our results in Sections \ref{section:dimension-monotonicity} and \ref{section:dimension-monotonicity-weak}. 
We consider corollaries of our results applied to special solutions of MCF, like self-shrinkers and immersed MCFs, and we point out fattening criteria as a consequence of Theorem \ref{thm:dim-mono-nonfattening}.

We first observe that if two smooth embedded MCFs intersect at some time, then the intersection set must have positive codimension two Hausdorff measure for all previous times.

\begin{cor}
\label{cor:int-preserved}
    Under the same assumptions as Theorem \ref{theorem main}, if $M_T \cap N_T \neq \emptyset$ for some time $T$, then $0<\cH^{n-1}(M_t \cap N_t)< \infty$ for all $t < T$. 
\end{cor}

In particular, one could apply Corollary \ref{cor:int-preserved} to intersecting ancient solutions, which means that they must have a large intersection set for all times $t \in (-\infty, T)$. 
A similar result follows for the self-intersection set of closed, immersed self-shrinkers.
We note that there are many examples of immersed self-shrinkers; see, for example, \cite{Drugan2017}.

\begin{cor}\label{cor:ShrinkerSelfInt}
    Let $F: M^n \to \bR^{n+1}$ be a closed, immersed self-shrinker. 
    If the self-intersection set $S(F)$ is nonempty, then $0<\cH^{n-1}\pr{S(F)}<\infty$. 
\end{cor}

\begin{proof}
    Since $F$ is self-shrinking, $F_t$ evolves by time-dependent scalings of $F$. 
    Since $S(F)$ is nonempty, we have that the self-intersection set $S_t= S(F_t)$ is also nonempty for all $t<0$. 
    By Theorem \ref{thm:self-int-main}, the fact that $S_t$ is nonempty for $t<0$ implies that $0<\cH^n(S_t)< \infty$ for all $t<0$.
\end{proof}

Theorem \ref{thm:self-int-main} says that if the self-intersection set of an immersed hypersurface has codimension strictly greater than two, then the mean curvature flow instantaneously makes the immersion an embedding. 
On the other hand, if the immersion cannot be embedded for topological reasons, we find interesting behavior of the self-intersection set over the flow, as described in the following corollary.

\begin{cor}\label{cor:ImmNotEmbFlow}
    Let $M$ be a smooth $n$-manifold which may be immersed but not embedded in $\bR^{n+1}$. 
    Let $F_0 = F: M \to \bR^{n+1}$ be a smooth immersion, and let $F_t$ be the mean curvature flow starting at $F_0$. 
    If $F_t$ exists for $t \in [0, T)$, then the self-intersection set $S_t$ satisfies $0<\cH^{n-1}\pr{S_t}< \infty$ for $t \in (0,T)$.
\end{cor}

\begin{proof}
Since $M$ may not be embedded in $\bR^{n+1}$, we have that $S_t \neq \emptyset$ for $t \geq 0$. 
By Theorem \ref{thm:self-int-main}, this implies that $0<\cH^{n-1}(S_t)< \infty$ for all $t>0$ when $F_t$ exists.
\end{proof}

In particular, Corollary \ref{cor:ImmNotEmbFlow} implies that the MCF of a smoothly immersed $\bR \bP^2$ in $\bR^3$ must satisfy $\dim(S_t) = 1$ and $0<\cH^1(S_t) < \infty$ immediately after the initial time and for all time up until the first singular time.

We now point out a reinterpretation of Theorem \ref{thm:dim-mono-nonfattening} as a criterion for fattening. 

\begin{thm}\label{theorem:FatteningCriterion}
    Let $M_t$ be a level set flow starting from a smooth closed embedded hypersurface $M$ in $\bR^{n+1}$. 
    Suppose that $M_t$ is smooth for $t \in [0, T)$ and that for some $\eps>0$, the inner and outer flows of $M_t$ have finitely many singularities for $t \in [0, T+\eps)$. 
    
    Suppose there exists a smooth closed embedded MCF $\{N_t\}_{t \in [T,T^*)}$ such that $N_t \not\sbst M_t$, $\dim(M_T \cap N_T) < n-1$, and $M_{t_0} \cap N_{t_0} \neq \emptyset$ for some $t_0 \in (T, \min(T^*, T+\eps))$. 
    Then, $M_t$ is fattening.
\end{thm}

\begin{proof}
By Definition \ref{def:lsf-sing}, a level set flow has finitely many singularities if the inner and outer flows, as the supports of integral unit regular Brakke flows, have finitely many singularities. 
By assumption, the level set flow $M_t$ has finitely many singularities for $t \in [0, T+\eps)$. Therefore, we may apply Theorem \ref{thm:dim-mono-nonfattening} to this time interval. 

Suppose that $M_t$ is non-fattening for $t \in [0, T+\eps)$. 
Then, it follows from $\dim(M_T \cap N_T) < n-1$ that $M_s\cap N_s = \emptyset$ for all $s \in (T, T+\eps)$. If $M_{t_0} \cap N_{t_0} \neq \emptyset$, then we find a contradiction. 
This implies that $M_t$ is fattening.
\end{proof}

A recent result of Chodosh--Daniels-Holgate--Schulze says that the assumption on the inner and outer flows in Theorem \ref{theorem:FatteningCriterion} holds for an isolated conical singularity in low dimensions~\cite[Theorem~1.2]{CDS23}. 
In fact, they give a fattening criterion for conical singularities: loosely, a level set flow with isolated conical singularities is fattening if and only if the corresponding conical singularity model fattens. 
Theorem \ref{theorem:FatteningCriterion} is akin to a generalization of their fattening criterion since it says, loosely, that a level set flow will fatten if it desingularizes isolated but locally disconnected singularities by a smooth connected flow.

We may go further with the fattening criterion using Theorem \ref{thm:loc-LSF-mono}. 
We find a more general fattening criterion if we make the additional assumption that the multiplicity one conjecture holds and also include minor restrictions on the behavior of the comparison MCF. 

\begin{thm}\label{theorem:FatteningCriterionGeneral}
Let $M_t$ be a level set flow starting from a smooth closed embedded hypersurface $M$ in $\bR^{n+1}$, and suppose that $M_t^{\rm out}$ and $M_t^{\rm in}$ have no higher multiplicity planar tangent flows.

Suppose there exists a smooth closed embedded $\{N_t\}_{t \in [T,T^*)}$ such that $N_t \not\sbst M^{\rm in}_t, M^{\rm out}_t$ and no connected component of $M^{\rm in}_t$ or $M^{\rm out}_t$ contains $N_t$. 
If $\dim(M_T \cap N_T) < n-1$ and $M_{t_0} \cap N_{t_0} \neq \emptyset$ for some $t_0 \in (T, T^*)$, then $M_t$ is fattening.
\end{thm}

Theorem \ref{theorem:FatteningCriterionGeneral} follows immediately from Theorem \ref{thm:loc-LSF-mono}, just as for Theorem \ref{theorem:FatteningCriterion}. 
As we have already remarked, the assumption regarding the absence of higher multiplicity planar tangent flows is a natural condition.
See the comments preceding Theorem~\ref{theorem-splittable-implies-GAP}.


\subsection{Failure of measure monotonicity}

In this section, we will give examples of mean curvature flows in $\bR^{n+1}$ whose intersection has increasing $(n-1)$-dimensional Hausdorff measure. 
This contrasts with simple examples of flows where the intersection has decreasing $(n-1)$-dimensional Hausdorff measure, e.g., the flows of two round spheres appropriately arranged in $\bR^{n+1}$. This shows that the $(n-1)$-dimensional Hausdorff measure of the intersection of MCFs does not satisfy any obvious monotonicity.

In our first result, we give an example of a mean curvature flow resembling a growing catenoidal neck whose intersection with a hyperplane has increasing $\cH^{n-1}$-measure. 

\begin{prop}\label{proposition increasing Hn-1 all time}
    There exist mean curvature flows $M^n_t$ and $N^n_t$ with uniformly bounded curvature in $\bR^{n+1}$ such that $t\mapsto \cH^{n-1}(M^n_t \cap N^n_t)$ is strictly increasing for all time. 
\end{prop}

\begin{proof}
The second named author and Mramor constructed $O(n)\times O(1)$-invariant eternal solutions $M_t$ which are noncompact, have a sign on mean curvature and uniformly bounded curvature for all time, and limit to a catenoid as $t\to -\infty$\cite{MramorPayne}. 
Since $M_t$ is $O(n)\times O(1)$-invariant, it can be represented as a flow of profile curves $\ga_t$ in $\bR^2$ which are reflection symmetric about the $x$-axis, where the rotation axis is represented by the $y$-axis (see~\cite[Figure 1]{MramorPayne}). 
Now, consider a hyperplane $N \subseteq \bR^{n+1}$ which is invariant with respect to the same $O(n)\times O(1)$ action as $M_t$, so that it is represented by the $x$-axis in $\bR^2$, orthogonal to the rotation axis. 
Since $N$ is minimal, $N_t = N$ for all $t$. 
Then, $M_t \cap N_t = M_t \cap N$ is a round $(n-1)$-sphere embedded in $\bR^{n+1}$, and $M_t \cap N_t$ is represented by a point on the $x$-axis of $\bR^2$ at position $x=r(t)$. 
In~\cite[Theorem 5]{MramorPayne}, it is shown that $M_t$ has strictly increasing distance from the axis of rotation over time, and the closest point of $\ga_t$ to the axis of rotation is $r(t)$ by the $O(1)$ reflection symmetry. 
This means that $r(t)$ is monotonically increasing in time. 
Since $r(t)$ corresponds to the radius of the $(n-1)$-sphere $M_t \cap N_t$, we have that $t \mapsto \cH^{n-1}(M_t \cap N_t)$ is monotonically increasing for all $t \in (-\infty, \infty)$.
\end{proof}

In Proposition~\ref{proposition increasing Hn-1 all time}, the two flows are both noncompact. 
In the next result, we will construct an example where compact mean convex flows can have increasing $(n-1)$-dimensional Hausdorff measure of their intersection for a short time. 
These will be constructed using the flow of an $O(n)\times O(1)$-invariant $S^{n-1} \times S^{1} \subseteq \bR^{n+1}$, resembling a perturbation of the round marriage ring.

\begin{prop}
\label{proposition increasing Hn-1 short time}
There exist compact, mean convex mean curvature flows $M^n_t$ and $N^n_t$ in $\bR^{n+1}$ such that for some $\de>0$, $t \mapsto \cH^{n-1}(M^n_t \cap N^n_t)$ is strictly increasing for $t \in [0, \de]$. 
\end{prop}

\begin{proof}
Consider a smooth closed, strictly convex curve $\ga$ in $\bR^2$ which is reflection symmetric with respect to the $x$-axis. 
We parametrize $\ga$ counterclockwise. 
Since $\ga$ is convex, it may be parametrized by the angle its tangent vector makes with the positively-oriented $x$-axis, that is, $\ga: [0, 2\pi) \to \bR^2$. 
Since $\ga$ is reflection symmetric, smooth, and convex, it must have two intersection points with the $x$-axis $\{y=0\}$. 
Moreover, these intersection points with the $x$-axis must have tangent lines orthogonal to the $x$-axis due to reflection symmetry, which implies that 
\begin{align*}\ga([0, 2\pi)) \cap \{y=0\} = \{r_{\min}, r_{\max}\},\end{align*}
where we define $r_{\min} := \ga(\frac{3\pi}{2})$ and $r_{\max} := \ga(\frac{\pi}{2})$. 
The points $r_{\min}$ and $r_{\max}$ are the closest and farthest points on $\gamma$ from the $y$-axis, respectively. 
By abuse of notation, we may identify $r_{\min}$ and $r_{\max}$ with their $x$-values on the $x$-axis $\{y=0\}$.

Now, we let $n>1$, and we construct a smooth closed, strictly convex curve $\ga: [0, 2\pi) \to \bR^2$, implicitly depending on $n$, with the following properties:
\begin{enumerate}
\item\label{item 1} $\ga$ is reflection symmetric with respect to the $x$-axis and forms a subset of $\{x>0\}.$
\item\label{item 2} $r_{\min}=10n$ and $r_{\max} - r_{\min} = 1$.
\item\label{item 3} We parametrize $\ga$ counterclockwise and have that
\begin{align*}
    \inf_{\te \in [\pi, 2\pi)}\ka(\te) \geq \frac{1}{5}.
\end{align*}
\item\label{item 4} The maximum (resp.\ minimum) of the curvature is achieved at $r_{\min}$ (resp.\ $r_{\max}$), and 
\begin{align}\label{equation min of curvature rmin}
    \sup_{\te \in [0, 2\pi)} \ka(\te)  = \ka\Big(\frac{3\pi}{2}\Big)= 10,
\end{align}
\begin{align}\label{equation max of curvature rmax}
    \inf_{\te \in [0, 2\pi)} \ka(\te) = \ka\Big(\frac{\pi}{2}\Big)= \frac{1}{10}.
\end{align}
\end{enumerate}
In our setup, $\ga([\pi, 2\pi)])$ is the subset of the curve bounded by points of $\gamma$ with horizontal tangent lines, and this curve segment is closer to the $y$-axis than the curve segment $\ga([0, \pi))$. 
Therefore, \eqref{item 3} is saying that the curvature of the curve segment $\ga([\pi, 2\pi))$ is lower bounded by $\frac{1}{5}$.
Since ~\eqref{item 4} does not strongly restrict the curvature away from the points $r_{\min}$ and $r_{\max}$ and since~\eqref{item 3} is a very weak curvature restriction, we can find a smooth closed convex curve $\ga$ satisfying the above conditions.

\begin{center}
\captionsetup{type=figure}
    \includegraphics[width=9cm]{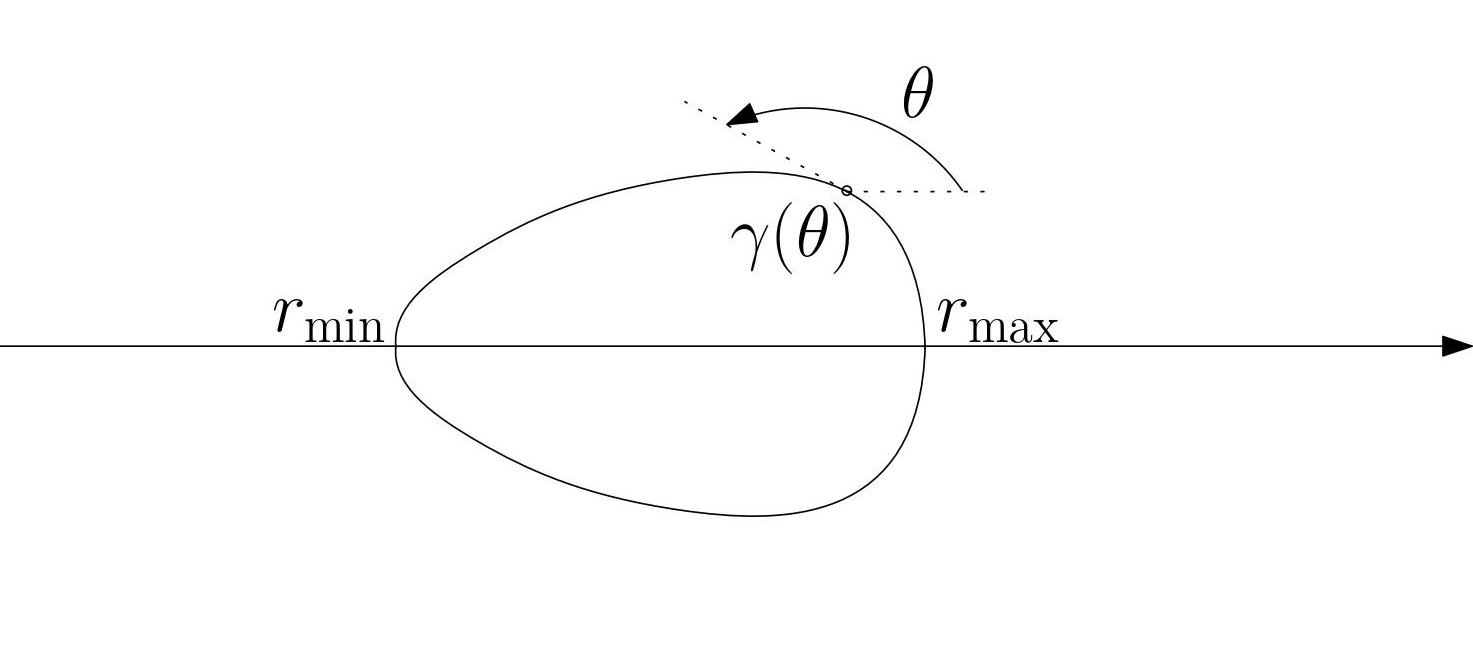}
    \captionof{figure}{An example of a profile curve satisfying the desired properties. 
    The curve is far away from the rotation axis so that the principal curvatures in those directions are small.}
\end{center}

Given $\ga$ satisfying these properties, we may consider it as the profile curve of an $O(n)\times O(1)$-invariant surface $M$ embedded in $\bR^{n+1}$. 
Topologically, $M$ is diffeomorphic to $S^{n-1} \times S^{1}$. 
The mean curvature of $M$, with respect to the inward-pointing normal, is constant with respect to the $S^{n-1}$ parameters and is a function only of $\te$ (see~\cite[\S 3]{BLT1}~\cite{AAG95}, which we adjust to our convention of rotating around the $y$-axis):
\begin{align}\label{equation H in rot sym}
    H(\te) = \ka(\te) + \frac{(n-1) \sin\te}{\ga_1(\te)},
\end{align}
where $\ga(\te) = (\ga_1(\te), \ga_2(\te)).$

Now, we will check that $M$ is mean convex. 
By the strict convexity of $\ga$, $\ka(\te)>0$ for all $\te \in [0, 2\pi)$. 
For $\te \in [0, \pi)$, $\sin\te \geq 0$, so $H(\te) >0$ in this case. 
For $\te\in [\pi, 2\pi)$, we apply condition~\eqref{item 3} and the fact that $\ga_1(\te) \geq r_{\min} = 10n$:
\begin{align*}
    H(\te) &\geq \frac{1}{5} + \frac{(n-1)\sin\te}{10n}>0.
\end{align*}

Since $M$ is a smooth closed surface in $\bR^{n+1}$, its mean curvature flow $M_t$ exists for $t \in [0, T)$ for some $T>0.$ 
Since the flow preserves rotational symmetry, $M_t$ will be represented by a profile curve $\ga_t$ in $\bR^2$ which is reflection symmetric with respect to the $x$-axis. 
The velocity of the flow $\ga_t$ is $H(\te)\nu_t$, where $\nu_t$ is the inward-pointing normal of $\ga_t$. 
By construction of $\ga$ and the fact that this is a smooth flow, there is $\eps>0$ such that $\ga_t$ will have two distinct intersection points with the $x$-axis for $t \in [0, \eps)$:
\begin{align}\label{equation intersection of gat with x}
\ga_t([0, 2\pi)) \cap \{y=0\} = \{r_{\min}(t), r_{\max}(t)\},
\end{align}
where $r_{\min}(0) = r_{\min}$ and $r_{\max}(0) = r_{\max}$. 
Moreover, by the reflection symmetry, $\dd_\te\ga_t(\frac{\pi}{2})$ and $\dd_\te\ga_t(\frac{3\pi}{2})$ will be orthogonal to the $x$-axis, so $r_{\min}(t) = \ga_{t}(\frac{3\pi}{2})$ and $r_{\max}(t) = \ga_t(\frac{\pi}{2})$. 
In particular, $\nu_t$ at $\te = \frac{\pi}{2}, \frac{3\pi}{2}$ will be parallel to the $x$-axis in $\bR^2$. 
This implies that 
\begin{align*}
    r_{\min}(t) &= r_{\min} + t H\Big(\frac{3\pi}{2}\Big) + O(t^2),\\
    r_{\max}(t) &= r_{\max} - t H\Big(\frac{\pi}{2}\Big) + O(t^2),
\end{align*}
when $t$ is small where $H(\frac{3\pi}{2})$ and $H(\frac{\pi}{2})$ denote the mean curvature of $\ga$ at $r_{\min} = r_{\min}(0)$ and $r_{\max} = r_{\max}(0)$, respectively. 
Applying~\eqref{equation min of curvature rmin} and~\eqref{equation max of curvature rmax} to~\eqref{equation H in rot sym} and using condition~\eqref{item 2},
\begin{align}\label{equation drmin}
    \frac{dr_{\min}(t)}{dt}\Big|_{t=0} 
    & = 10 - \frac{(n-1)}{10n}\geq 9\\\label{equation drmax}
    \frac{dr_{\max}(t)}{dt}\Big|_{t=0} 
    &= - \pr{\frac{1}{10} + \frac{(n-1)}{10n +1}} \geq -\frac{1}{5}.
\end{align}

Now, let $P \subseteq  \bR^{n+1}$ be the hyperplane corresponding to the $x$-axis profile curve $\{y=0\}$ in $\bR^2$. 
Since $P$ is minimal, the mean curvature flow $P_t$ starting from $P$ satisfies $P_t = P$ for all $t \in \bR$. 
Using~\eqref{equation intersection of gat with x}, $M_t \cap P_t$ are two round $(n-1)$-spheres in $\bR^{n+1}$ of radius $r_{\min}(t)$ and $r_{\max}(t)$, so
\begin{align*}
    \cH^{n-1}(M_t \cap P_t) = C_{n-1}r_{\min}(t)^{n-1} + C_{n-1} r_{\max}(t)^{n-1},
\end{align*}
where $C_{n-1}$ is the $(n-1)$-dimensional area of a unit $(n-1)$-sphere.
Using~\eqref{equation drmin} and~\eqref{equation drmax},
\begin{align}\label{equation measure increasing}
    \frac{d}{dt}\cH^{n-1}(M_t \cap P_t)\Big|_{t=0} &\geq (n-1)C_{n-1}\Big(9 r_{\min}^{n-2} - \frac{1}{5}r_{\max}^{n-2}\Big)\nonumber\\
\intertext{\small Using condition~\eqref{item 2},}
    &= (n-1)C_{n-1}\Big(9 (10n)^{n-2} - \frac{1}{5}(10n+1)^{n-2}\Big)\nonumber\\
    &= (n-1) (10n)^{n-2} C_{n-1}\Big(9 - \frac{1}{5}\Big(1+\frac{1}{10n}\Big)^{n-2}\Big)\nonumber\\
\intertext{\small Since $\Big(1+\frac{1}{10n}\Big)^{n-2}\leq \rm e$ for all $n\geq 2$,}
    &>0.
\end{align}
By~\eqref{equation measure increasing}, we conclude that there exists $\de>0$ such that $t \mapsto \cH^{n-1}(M_t \cap P_t)$ is increasing for $t \in [0, \de]$. 

Now, let $R \geq 1$ such that $M_t \subseteq  B_{R/2}$, and suppose $N$ is a closed, mean convex hypersurface in $\bR^{n+1}$ such that $N \cap B_R = P \cap B_R$. 
By the pseudolocality of mean curvature flow (see~\cite{Chen2007, IlmanenNevesSchulzeNetworkFlow}), for each $\eps>0,$ there is $\de>0$ such that $\sup_{B_{R/2}}|H_{N_t}| \leq \eps$ for all $t \in [0, \de]$, where $H_{N_t}$ is the mean curvature of $N_t$. 
For $\eps$ small enough, $M_t \cap N_t$ will consist of two embedded, nearly round $(n-1)$-spheres in $\bR^{n+1}$ of radius approximately $r_{\min}(t)$ and $r_{\max}(t)$. 
By a continuity argument,
\begin{align*}
\frac{d}{dt}\cH^{n-1}(M_t \cap N_t)\Big|_{t=0}  = \frac{d}{dt}\cH^{n-1}(M_t \cap P_t)\Big|_{t=0} - C(\eps,n),\end{align*}
where $C(\eps, n)$ is some function of $\eps$, $n$, and the geometry of $M$ such that $\lim_{\eps \to 0} C(\eps, n) = 0$. 
Then, choosing $\eps>0$ small enough, we have that there is $\de>0$ such that $t \mapsto \cH^{n-1}(M_t \cap N_t)$ is increasing for $t \in [0,\de]$. 
Note that $N_t$ is not strictly mean convex at $t=0$, but it becomes so for $t \in (0, \de]$.  
\end{proof}


\subsection{Failure of nodal domain monotonicity}
In this section, we will give examples of smooth flows in $\bb R^{n+1}$ whose intersection has an increasing number of connected components.
We also consider the number of connected components of $M_t \setminus N_t$, which can be interpreted as the number of nodal domains. We find that counting connected components associated to the intersection of MCFs does not provide any obvious monotonicity.

\begin{prop}\label{proposition nodal domain monotonicity failure}
    There exist compact mean convex mean curvature flows $M_t^n$ and $N_t^n$ in $\bb R^{n+1}$ such that for some $\delta>0,$ neither the number of components of $M_t\cap N_t$ nor the number of components of $M_t\setminus N_t$ is non-increasing for $t\in[0,\delta].$
\end{prop}

\begin{proof}
    An example for this proposition comes from Grayson's dumbbell~\cite{Grayson89} intersecting with a correctly positioned plane.
    By a similar pseudolocality argument as that given in the proof of Proposition~\ref{proposition increasing Hn-1 short time}, we can replace the plane with a large compact mean convex hypersurface.

     We will demonstrate this example when $n=2.$ The construction for $n>2$ follows similarly. 
 Let $\cT$ be the genus one shrinking torus constructed by Angenent~\cite{Angenent92}.
 We take 
\begin{align*}
L&:=\sup_{X\in T}|X| + 100 <\infty,\text{ and}\\
\varepsilon &:= \inf_{X\in \cT}|X|>0.
\end{align*}
 
We are going to construct a closed mean convex surface following the idea of Grayson~\cite{Grayson89}. 
 To be more specific, we construct a function $f\colon [-5L-10, 5L+10]\to\bb R_{\ge 0}$ such that the following properties hold.
 \begin{enumerate}
 	\item\label{i1} 
 	$f(x) > \sqrt{L^2 - (x-4L)^2}$ for $x\in [3L, 5L].$
 	\item\label{i2}
 	$f(x) > \sqrt{L^2 - (x+4L)^2}$ for $x\in [-5L, -3L].$
 	\item\label{i3}
 	$f(x)\ge \varepsilon$ for $x\in [-3L, -L]\cup [L, 3L]$ and
 	$f(x)=\varepsilon$ for $x\in [-L, L].$
 	\item\label{i-inc}
 	$f$ is strictly increasing on $[-5L-10, -4L]\cup [L,4L]$ and strictly decreasing on $[-4L, -L]\cup [4L, 5L+10].$
 	\item\label{i4}
 	$f(-5L-10) = f(5L+10)=0,$ so viewing $[-5L-10, 5L+10]$ as a subset of the $x$-axis, we can rotate the graph of $f$ with respect to the $x$-axis to get a mean convex surface with $O(2)\times O(1)$-invariance in~$\bb R^3.$
 \end{enumerate}
 
 To see that this is possible, first, we can assume $f$ is an even function so that the $O(1)$-symmetry in \eqref{i4} is achieved.
 By the formula of curvatures of rotational surfaces \eqref{equation H in rot sym}, the mean convexity condition in \eqref{i4} can be achieved by choosing $f$ such that the sign of the curvature of the graph of $f$ remains the same when $x$ is in $(-3L+10,-L-10)$ and $(L+10, 3L-10)$ using smooth enough cutoffs.
 Since these do not strongly restrict the curvature away from these intervals, we can find a smooth function $f$ such that the conditions above are satisfied.
 We define the resulting rotating surface to be $M.$

\begin{center}
\captionsetup{type=figure}
    \includegraphics[width=12cm]{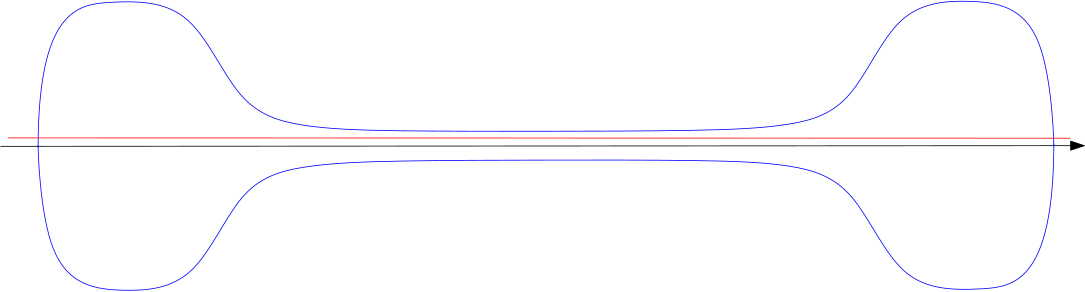}
    \captionof{figure}{The blue curve is an example of a profile curve satisfying the desired properties. 
    The surface obtained by rotating the blue curve contains two round spheres $S_1$ and $S_2$ in the enclosed region, and its neck part can be surrounded by a shrinking torus.
    In a short time, the number of components of its intersection with the plane, obtained by rotating the straight red line around the $y$-axis, will increase from one to two.
    }
\end{center}

By conditions \eqref{i1} and \eqref{i2}, we can choose two round spheres 
\begin{align*}
     S_1 &= \dd B_{L}\pr{(-4L, 0, 0)}\text{ and}\\
     S_2 &= \dd B_{L}\pr{(4L, 0, 0)}
\end{align*}
contained in the bounded region enclosed by $M.$
We position a closed genus-one self-shrinker $\cT$ constructed by Angenent~\cite{Angenent92} such that its rotation axis is the $x$-axis; that is, $\cT$ is positioned to be reflection-symmetric across the $y$-axis.\footnote{Recall that a self-shrinker satisfies $H=\frac{\pair{x,\N}}{2},$ so we can only rotate it instead of applying an arbitrary translation.} 
Thus, by \eqref{i3}, this implies 
\begin{align*}
\cT\cap M = \dd B_\varepsilon \cap \set{x=0}.
\end{align*}
 This intersection is tangential and $\cT$ lies on one side of $M,$
 so by the proof of Theorem~\ref{theorem main} in the case when the intersection of two flows is of lower dimension, we know that the flows of $M$ and $\cT$ will become disjoint after the initial time.
 
Now, we let $\cT(t), S_1(t), S_2(t),$ and $M_t$ be the smooth mean curvature flows starting from $\cT,S_1,S_2,$ and $M.$
Note that $\cT(t), S_1(t),$ and $S_2(t)$ are all self-shrinking flows, so
 \begin{align}\label{T-evol}
 \cT(t) &= \sqrt{1-t}\,\cT,\\
 S_1(t) &= \dd B_{\sqrt {L^2-4t}}\pr{(-4L,0,0)},\text{and }\nonumber\\
 S_2(t) &= \dd B_{\sqrt {L^2-4t}}\pr{(4L,0,0)}\nonumber
 \end{align}
 for $t\in[0,1).$
Since $M$ is compact, there exists $\delta\in(0,1)$ such that $M_t$ is smooth for $t\in [0,\delta].$
Thus, if we take $P = \set{z = \sqrt{1-\delta/2}\cdot \varepsilon},$ we get 
 \begin{align}\label{M0-one-component}
 P\cap M_0 \text{ has one component}
 \end{align}
 based on \eqref{i1}, \eqref{i2}, \eqref{i3}, and \eqref{i-inc}.
 However, at time $t=2\delta/3,$ the avoidance principle and the evolution equations~\eqref{T-evol} imply 
 $P\cap \pr{M_t\cap \set{x=0}}=\emptyset$ but 
 $P\cap \pr{M_t\cap \set{x=\pm 4L}}\neq \emptyset$
 based on \eqref{i1}, \eqref{i2}, and~\eqref{i3}.
 These imply 
 \begin{align}\label{Mt-two-components}
 P\cap M_t \text{ has at least two components}.
 \end{align}
 when $t=2\delta/3.$
 Combining \eqref{M0-one-component} and \eqref{Mt-two-components}, we find the desired example, given by $M_t$ intersecting with a plane.
 
 We can find a closed mean convex flow $N_t$ such that $N_0\cap B_{R} = P\cap B_{R}$ for any $R>0.$
 Based on the pseudolocality argument in the proof of Proposition~\ref{proposition increasing Hn-1 short time}, by taking $R$ large, we can derive that the number of components of $M_t\cap N_t$ or $M_t\setminus N_t$ is not non-increasing.
\end{proof}


\subsection{Failure of dimension monotonicity for general Brakke flows}
\label{sec:mono-counterexample}

We will give examples of Brakke flows in $\bb R^{n+1}$ whose intersection has increasing Hausdorff dimension for a short time (see Figure~\ref{figure-cone}).
These examples show that a general dimension monotonicity result for intersecting Brakke flows cannot be true.
One ingredient is precise information on inner and outer flows starting from a hypersurface with an isolated conical singularity, proven by Chodosh--Daniels-Holgate--Schulze \cite{CDS23}.
\begin{thm}[\!\mbox{\cite[Theorem~4.1]{CDS23}}]
\label{thm:CDS}
    For $2\le n\le 6,$ suppose $M$ is a hypersurface in $\bb R^{n+1}$ with an isolated singularity modeled on a regular cone $\cC.$
    Then the innermost and outermost flows starting from $M$ are modeled on the innermost and outermost expanders of $\cC$ near $({\bf 0}, 0)\in\bb R^{n+1}\times\bb R.$
\end{thm}

\begin{cor}\label{cor:bad-ex-1}
    Let $2\le n\le 6.$
    There exist a compact hypersurface $M$ in $\bb R^{n+1}$ with an isolated singularity and a Brakke flow $\cM_t$ with $\cM_0 = \cH^n\lfloor M$ such that $\dim\pr{\spt \cM_t\cap P_t}$ is not non-increasing for $t\in [0,\varepsilon]$ where $\varepsilon$ is a positive number and $P_t = P$ is a static flow of planes.
\end{cor}

\begin{proof}
	Let $\cC\sbst \bb R^{n+1}$ be a regular cone such that the level set flow starting from it fattens.
	There are many explicit examples of fattening cones, including double cones with wide angles~\cite{angenentilmanenchopp95}.
    Let $\cC_t$ be the outermost flow starting from $\cC,$ which is modeled on an expander $\Sigma$ (see \cite{I95, CCMS, CDS23}).
	
	We construct $M$ by compactifying $\cC$ so that $M\cap B_R = \cC\cap B_R$ for some large $R>0.$
	By Theorem~\ref{thm:CDS}, the outermost flow $\cM_t$ starting from $M$ is modeled on $\mathcal C.$
	Then, we can choose a plane $P$ and a small positive number $\varepsilon$ such that
	\begin{align*}
	\dim(\spt\cM_t \cap P) = n-1
	\end{align*}
	for $t\in (0,\varepsilon)$
	but $\spt \cM_0\cap P
	= \spt M \cap P
	= \{0\}.$
	Thus, the dimension monotonicity fails for the intersection of $\spt \cM_t$ and $P.$
\end{proof}

\begin{rmk}\label{rmk:smoothBFIntersectionPrinciple}
We make a final remark based on the recent works~\cite{IW24, Ket24}, which give a way for Brakke flows starting from smooth closed hypersurfaces to violate the intersection principle.
Ilmanen--White~\cite{IW24} and Ketover~\cite{Ket24} independently showed that when $g$ is large, there exists an asymptotically conical shrinker $\Sigma_g$ that has genus $g$ and two ends such that the level set flow starting from $\Sigma_g$ fattens.
By Ilmanen--White \cite{IW24} or Lee--Zhao~\cite{LZ24} (see \cite{white02, AIV, Ket24}), there exists a closed embedded smooth surface $M_g\sbst\bb R^3$ such that the flow starting from it develops a singularity modeled on $\Sigma_g,$ say at the origin at time $t=1.$
Based on \cite{CS21}, $M_1$ is a singular hypersurface with a conical singularity at the origin.
By Theorem~\ref{thm:CDS}, the outermost flow $\cM_t$ starting from $M_1$ is modeled on the asymptotic cone of $\Sigma.$
Thus, one may get a version of Corollary~\ref{cor:bad-ex-1} for Brakke flows starting from a smooth closed surface using these constructions.
\end{rmk}

\bibliographystyle{amsalpha}
\bibliography{bibliography}

\end{document}